\documentclass[11pt]{amsart}

\usepackage[left=2cm,right=2cm,top=3cm,bottom=3cm]{geometry}

\usepackage{amsmath, amssymb, latexsym, amsthm}
\usepackage{fullpage}
\usepackage{color}
\usepackage{graphicx}
\usepackage{stmaryrd,amsbsy,enumerate}

\def\COMMENT#1{}
\let\COMMENT=\footnote 

\newtheorem{theorem}{Theorem} 
\newtheorem{theorem*}{Theorem} 
\newtheorem{proposition}[theorem]{Proposition} 

\newtheorem{corollary}[theorem]{Corollary}

\theoremstyle{definition}

\newtheorem{question}[theorem]{Question}

\theoremstyle{remark}
\newtheorem{claim}{Claim}

{\end{enumerate}}

\renewcommand{\Pr}{\mathbb{P}}

\newcommand{\eps}{\varepsilon}
\renewcommand{\epsilon}{\varepsilon}
\DeclareMathOperator{\im}{im}

\newcounter{lth}
\newcounter{2th}
\newcounter{3th}
\newcounter{4th}
\newcounter{5th}
\title{Monochromatic triangles in three-coloured graphs}
\author{James Cummings$^\fnsymbol{lth}$} \thanks{$^\fnsymbol{lth}$Department of Mathematical Sciences, Carnegie Mellon University, Pittsburgh PA 15213, USA.
E-mail: {\tt jcumming@andrew.cmu.edu} This author's research was partially supported by National Science Foundation
     grant DMS-1101156}
\author{Daniel Kr\'al'$^\fnsymbol{2th}$}\thanks{$^\fnsymbol{2th}$Computer Science Institute, Faculty of Mathematics and Physics, Charles University, Malostransk{\'e} n{\'a}m{\v e}st{\'\i} 25, 118 00 Prague, Czech Republic. E-mail: {\tt \{kral,treglown\}@iuuk.mff.cuni.cz} The work leading to this invention has received funding from the European Research Council under the European Union's Seventh Framework Programme (FP7/2007-2013)/ERC grant agreement no.~259385.}
\author{Florian Pfender$^\fnsymbol{3th}$} \thanks{$^\fnsymbol{3th}$Universit\"at Rostock, Institut f\"ur Mathematik, 18057 Rostock, Germany
and University of Colorado Denver, Mathematical and Statistical Sciences, Denver, CO, 80202, USA.
E-mail: {\tt florian.pfender@uni-rostock.de}}
\author{Konrad Sperfeld$^\fnsymbol{4th}$}\thanks{$^\fnsymbol{4th}$Universit\"at Rostock, Institut f\"ur Mathematik, 18057 Rostock, Germany 
E-mail: {\tt konrad.sperfeld@uni-rostock.de}}
\author{Andrew Treglown$^\fnsymbol{2th}$}
\author{Michael Young$^\fnsymbol{5th}$}\thanks{$^\fnsymbol{5th}$Department of Mathematics, Iowa State University, Ames, IA 50011, USA.
E-mail: {\tt myoung@iastate.edu} This author's research was supported by NSF grant DMS 0946431}

\begin{document}

\begin{abstract}
In 1959, Goodman~\cite{good} determined the minimum number
of monochromatic triangles in a complete graph whose edge set is
$2$-coloured. Goodman~\cite{goodDM} also raised the question of
proving analogous results for complete graphs whose edge sets
are coloured with more than two colours. In this paper,
for $n$ sufficiently large, we determine the minimum number
of monochromatic triangles in a $3$-coloured copy of $K_n$.
Moreover, we characterise
those $3$-coloured copies of $K_n$ that contain the minimum number of 
monochromatic triangles.
\end{abstract}

\maketitle

\section{Introduction}\label{intro}
The \emph{Ramsey number $R_k (G)$} of a graph $G$ is the minimum $n \in \mathbb N$ such that every $k$-colouring of $K_n$
contains a monochromatic copy of $G$.
(In this paper we say a graph $K$ is \emph{$k$-coloured} if we have coloured the \emph{edge set} of $K$ using $k$ colours.
Note that the edge colouring need not be proper.)
A famous theorem of Ramsey~\cite{ramsey}
asserts that $R_k (G)$ exists for all graphs $G$ and all $k \in \mathbb N$.

In light of this, it is also natural to consider the so-called \emph{Ramsey multiplicity} of a graph:
Let $k,n \in \mathbb N$ and let $G$ be a graph. The  \emph{Ramsey multiplicity $M_k (G,n)$} of $G$ is the minimum
number of monochromatic copies of $G$ over all $k$-colourings of $K_n$. 
(Here, we are counting unlabelled copies of $G$ 
in the sense that we count the number of distinct monochromatic subgraphs of $K_n$ that are isomorphic to $G$.)
In the case when $k=2$ we simply write $M(G,n)$.
The following classical result of Goodman~\cite{good} from 1959 gives the precise value of $M(K_3,n)$.
\begin{theorem}[Goodman~\cite{good}]\label{good}
Let $n \in \mathbb N$. Then 
\begin{equation*} 
 M(K_3,n)= \left\{\begin{array}{ll}
n(n-2)(n-4)/24 & \text{if $n$ is even;}\\
n(n-1)(n-5)/24 & \text{if $n \equiv 1\mod 4$;} \\
(n+1)(n-3)(n-4)/24 & \text{if $n \equiv 3\mod 4.$}
\end{array} \right.\end{equation*}
\end{theorem}
A graph $G$ is \emph{$k$-common} if $M_k (G,n)$ asymptotically equals, as $n$ tends to infinity, the expected number
of monochromatic copies of $G$ in a random $k$-colouring of $K_n$. 
Erd\H{o}s~\cite{E}
 conjectured that $K_r$ is $2$-common for every $r \in \mathbb N$. Note that Theorem~\ref{good} implies that this
conjecture is true for $r=3$. However, Thomason~\cite{AT, Th} disproved the conjecture in the case when $r=4$.
Further, Jagger, {\v{S}}{\v{t}}ov{\'{\i}}{\v{c}}ek, and Thomason~\cite{JST} proved that any graph $G$ that contains $K_4$ is not $2$-common. 
Recently, Cummings and Young~\cite{CuYo} proved that graphs $G$ that contain $K_3$ are not $3$-common.
The introductions of~\cite{CuYo} and~\cite{non} give  more detailed overviews of  $k$-common graphs.

The best known general lower bound on $M(K_r,n)$ was proved by Conlon~\cite{conlon}.
Some general bounds on $M_k (K_r,n)$ are given in~\cite{fox}.
See~\cite{burr} for a (somewhat outdated) survey on Ramsey multiplicities.

The problem of obtaining a $3$-coloured analogue of Goodman's theorem also has a long history. In fact, it is not entirely
clear when this problem was first raised.
In 1985, Goodman~\cite{goodDM} simply refers to it as ``an old and difficult problem''.
Prior to this, Giraud~\cite{giraud} proved that, for sufficiently large $n$, $M_3 (K_3, n) >4\binom{n}{3}/115$.
Wallis~\cite{wallis} showed that $M_3 (K_3,17) \leq 5$ and then, together with Sane~\cite{sane}, proved that $M_3 (K_3,17)=5$.
(Greenwood and Gleason~\cite{green} proved that $R_3 (K_3)=17$, therefore, $M_3 (K_3, 16)=0$.)

The focus of this paper is to give the exact value of $M_3(K_3,n)$ for sufficiently large $n$, thereby yielding a $3$-coloured analogue of Goodman's theorem. Moreover, we characterise
those $3$-coloured copies of $K_n$ that contain exactly $M_3(K_3,n)$ monochromatic triangles.

Given $n \in \mathbb N$ we define a special collection of $3$-coloured complete graphs on $n$ vertices, $\mathcal {G}_n$ as follows:
\begin{itemize}
\item Consider the (unique) $2$-coloured copy $K$ of $K_5$ on $[5]$ without a monochromatic triangle. Replace the vertices 
of $K$ with disjoint vertex classes  $V_1,\dots, V_5$ such that $||V_i|-|V_j||\leq 1$ for all $1\leq i,j \leq 5$ and 
$|V_1|+\dots+|V_5|=n$. For all $1\leq i\not = j \leq 5$, add all possible edges between $V_i$ and $V_j$ using the colour of  $ij$ in $K$.
For each $1\leq i \leq 5$, add all possible edges inside $V_i$ in a third colour. Denote the resulting complete $3$-coloured graph
by $G_{ex}(n)$ (see Figure~1).
\item $\mathcal G_n$ consists of $G_{ex}(n)$ together with all graphs obtained from $G_{ex}(n)$ by recolouring a (possibly empty) matching $M_{i,j}$ in $G_{ex}(n)[V_i,V_j]$ with the third colour for all $1\leq i\not = j \leq 5$, such that the recolouring does not introduce
any new monochromatic triangles (see Figure~1).
\end{itemize}

\begin{figure}\label{fig:Gn}
\includegraphics[bb=122 522 500 676,clip]{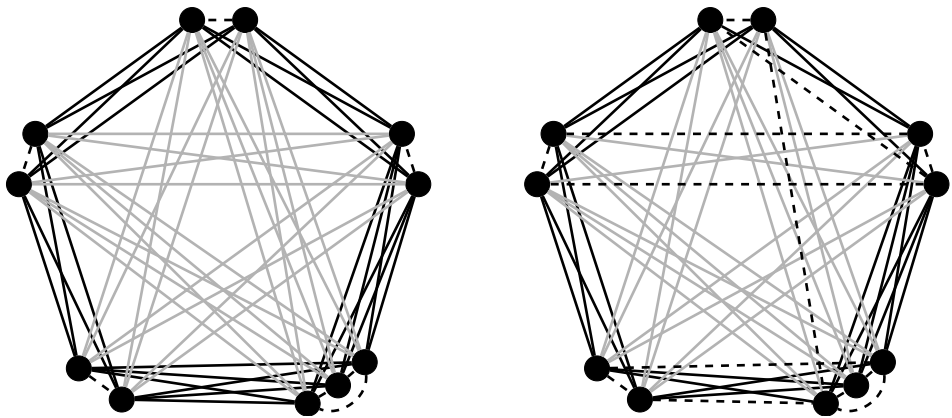}
\caption{$G_{ex}(11)$ and another element of $\mathcal G_{11}$}
\end{figure}

Notice that the graphs in $\mathcal G_n$ only contain monochromatic triangles of one colour. The following is our main result.
\begin{theorem}\label{mainthm}
There exists an $n_0 \in \mathbb N$ such that the following holds.
Suppose $G$ is a complete $3$-coloured graph on $n\geq n_0$ vertices which contains the smallest number of monochromatic triangles amongst
all complete $3$-coloured graphs on $n$ vertices. Then $G$ is a graph from $\mathcal G_n$. 
\end{theorem}
\begin{corollary} There exists an $n_0 \in \mathbb N$ such that the following holds. Suppose $n\geq n_0$ and write $n=5m+r$ where $m,r \in 
\mathbb N$ such that $0\leq r\leq 4$.
Then
$$M_3 (K_3, n)=r\binom{m+1}{3}+(5-r)\binom{m}{3}.$$
\end{corollary}
The proof of Theorem~\ref{mainthm} uses Razborov's method of flag algebras~\cite{FA} together with a probabilistic argument.

Goodman~\cite{goodDM} also raised the question of establishing $k$-coloured analogues of Theorem~\ref{good} for $k \geq 4$.
Let $k \geq 3$ and $n \in \mathbb N$. Fox~\cite{fox} gave an upper bound on $M_{k} (K_3,n)$ by considering the following graphs:
Set $m:=R_{k-1} (K_3)-1$. Consider a $(k-1)$-coloured copy $K$ of $K_m$ on $[m]$ without a monochromatic triangle. Replace the vertices 
of $K$ with disjoint vertex classes  $V_1,\dots, V_m$ such that $||V_i|-|V_j||\leq 1$ for all $1\leq i,j \leq m$ and 
$|V_1|+\dots+|V_m|=n$. For all $1\leq i\not = j \leq m$, add all possible edges between $V_i$ and $V_j$ using the colour of  $ij$ in $K$.
For each $1\leq i \leq m$, add all possible edges to $V_i$ using a $k$th colour. Denote the resulting complete $k$-coloured graph
by $G_{ex}(n,k)$. (Thus, $G_{ex}(n)=G_{ex}(n,3)$.)
\begin{question}Let $k\geq 4$ and $n\in \mathbb N$ be sufficiently large. 
Is $M_k (K_3,n)$ equal to the number of monochromatic triangles in $G_{ex} (n,k)$?
\end{question}

\section{Notation}

We will make the convention that the set of colours used in a $k$-colouring of the edges of a graph is
$[k]$. In the case of a $3$-colouring we will generally refer to the colours $1$, $2$ and $3$ as ``red'', ``blue'' and ``green''.  
When  $H$ and $H'$ are two $k$-coloured graphs, an {\em isomorphism} between them is a function
$f: V(H)  \rightarrow V(H')$ which is a graph isomorphism and respects the colouring. 
Two $k$-coloured graphs  $H$ and $H'$ are {\em isomorphic} ($H \cong H'$) if and only if there is an isomorphism between them. 

Given $r \in \mathbb N$, we denote the complete graph on $r$ vertices by $K_r$ and define  $R(r,r):=R_2 (K_r)$.
 Given $k$ and $c \in [k]$, we define $K^r_c$ to be the $k$-coloured complete graph in which every edge of $K_r$
  is given the colour $c$. We define  ${\mathcal K}^r$  to be $\{ K^r_c : c \in [k] \}$, that is to say the set
  of  monochromatic $K_r$'s. 
 Suppose $G$ is a $k$-coloured graph and let $v \in V(G)$ and $i \in [k]$. Then we will use  $N_i(v)$ to denote
the set of vertices in $G$ that receive an edge of colour $i$ from $v$.
 
For a graph $G$ and a vertex set $V \subseteq V(G)$, we denote by $G[V]$ the subgraph of $G$ induced by $V$.
 Given $v_1, \dots , v_m \in V(G)$ we write $G[v_1, \dots , v_m]$ for  $G[\{v_1, \dots , v_m\}]$,
 and for disjoint subsets $V$ and $W$ of $V(G)$ we denote by $G[V,W]$ the bipartite graph with vertex classes $V$ and $W$ whose
edge set consists of those edges between $V$ and $W$ in $G$.
 When $G$ is a $k$-coloured graph, we view $G[V]$ as a $k$-coloured graph with the edge colouring inherited
 from $G$, and do likewise for $G[v_1, \dots , v_m]$ and for $G[V,W]$. 

Throughout the paper, we write, for example, $0<\nu \ll \tau \ll \eta$ to mean that we can choose the constants
$\nu, \tau, \eta$ from right to left. More
precisely, there are increasing functions $f$ and $g$ such that, given
$\eta$, whenever we choose some $\tau \leq f(\eta)$ and $\nu \leq g(\tau)$, all
calculations needed in our proof are valid. 
Hierarchies with more constants are defined in the obvious way.
Finally, the set of all $k$-subsets of a set $A$ is denoted by $[A]^k$.

In the proof of Theorem~\ref{mainthm} we will omit floors and ceilings whenever this does not affect the
argument.

\section{Graph densities}

   From this point on we are exclusively concerned with $3$-colourings,  mostly colourings of complete
   graphs. Suppose  $H$ and $G$ are $3$-coloured complete graphs where $|H| \leq |G|$. 
   Let $d(H,G)$ denote the number of sets $V \in [V(G)]^{\vert H \vert}$ such that
   $G[V] \cong H$, and  define the \emph{density of $H$ in $G$} as
\[
   p(H, G) := \frac{d(H,G)}{\binom{|G|}{|H|}}.
\]
   This quantity has a natural probabilistic interpretation, namely it is the probability that
   if we choose a set ${\mathbf V} \in [V(G)]^{\vert H \vert}$ uniformly at random then
   ${\mathbf V}$ induces an isomorphic copy of $H$. 

   When $\mathcal H$ is a family of $3$-coloured complete graphs $H$ of some fixed size $k$ with  $k \leq |G|$,
   we define 
\[
   p(\mathcal H, G) := \sum_{H \in {\mathcal H}} p(H, G),
\]
  that is to say the probability that a random ${\mathbf V} \in [V(G)]^k$ induces a coloured graph isomorphic
  to an element of $\mathcal H$. In the sequel we generally write ``$H'$ is an $\mathcal H$''
  as an abbreviation for ``$H'$ is isomorphic to some $H \in {\mathcal H}$'',
 ``$G$ contains an $\mathcal H$'' as an abbreviation for ``$G$ contains an induced isomorphic copy of an element of $\mathcal H$'', and 
 ``an $\mathcal H$ in $G$'' for ``an induced copy of some element of $\mathcal H$ in $G$''.

   For $n \ge \vert H \vert$ we let
 $p^{\rm min}(H, n)$ be the minimum value of
 $p(H, G)$ over all $3$-coloured complete graphs $G$ on $n$ vertices. When
  $\mathcal H$ is a family of $3$-coloured complete graphs $H$ of some fixed size $k \le n$, we
 let  $p^{\rm min}({\mathcal H}, n)$ be the minimum value of $p({\mathcal H}, G)$
 over all $3$-coloured complete graphs $G$ on $n$ vertices.

  We now define a certain class $\mathcal H$ of ``bad'' $3$-coloured complete graphs on $4$ vertices. As motivation, we note
  that we are defining a set of $3$-coloured graphs $H$ such that $\max_{G \in {\mathcal G}_n} p(H, G) \rightarrow 0$ 
  with increasing $n$.

  Let ${\mathcal H}(i,j,k)$ be the class of $3$-coloured complete graphs on $4$ vertices 
  with  a monochromatic triangle, $i$ extra edges of that same colour, and $j$ and $k$ edges of the other colours,
  respectively (with $i+j+k=3$, $j\ge k$).
  Define $\mathcal H:=  {\mathcal H}(2,1,0) \cup {\mathcal  H}(1,1,1) \cup {\mathcal  H}(0,2,1)$.

\begin{figure}\label{fig:H}
\includegraphics[bb=132 582 475 676,clip]{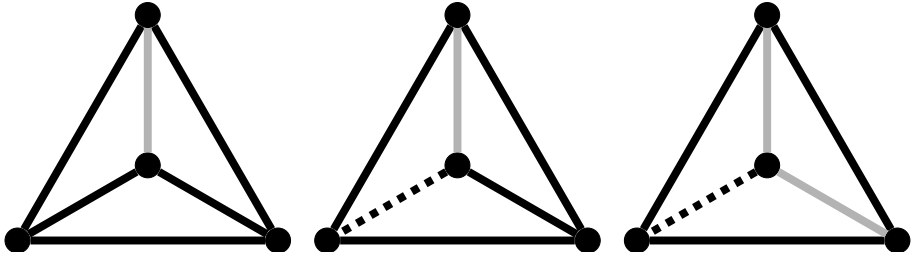} 
\caption{Representative elements of ${\mathcal H}(2,1,0)$, ${\mathcal H}(1,1,1)$ and ${\mathcal H}(0,2,1)$.}
\end{figure}

  The following result about graph densities will be used in the proof of Theorem \ref{mainthm}.
  It provides an (asymptotically) optimal lower bound on the density of monochromatic triangles,
  and also asserts that copies of colourings from the class $\mathcal H$ are rare in any
  colouring that comes close enough to achieving this bound. The proof is given in Section \ref{flags}.

\begin{proposition} \label{prop:flag}  

For all $\epsilon > 0$ there is $n_0$ such that for all $3$-coloured complete graphs $G$ on at least  $n_0$ vertices:

\begin{enumerate}

\item  \label{densityclaim}  $p({\mathcal K}^3, G) \ge 0.04 - \epsilon$. 

\item  \label{forbidclaim}  If $ p({\mathcal K}^3, G) \le 0.04$, then $p({\mathcal H}, G) < \epsilon$.
 
\end{enumerate} 

\end{proposition}

\section{Flag algebras} \label{flags}

   In this section we use the method of {\em flag algebras} due to Razborov \cite{FA} to prove
   Proposition \ref{prop:flag}.  The flag machinery described in subsections \ref{background} and \ref{averaging} is due
   to Razborov, as is the idea of using semidefinite programming for search of valid inequalities
   using this framework.

\subsection{Some background} \label{background} 

   We start by  describing how the main concepts of the general theory of flag algebras  look in the case  
    of $3$-coloured complete graphs. Let ${\mathcal M}_l$ be the set of isomorphism classes of $3$-coloured complete graphs
   on $l$ vertices. It is helpful  to know $\vert {\mathcal M}_l \vert$ for small values of $l$;
   computing this value is a classical enumeration problem \cite{oeis}, in particular $\vert {\mathcal M}_l \vert =  1, 1, 3, 10, 66, 792$
   for $l = 0,1,2,3,4,5$.  

   A {\em type} $\sigma$ is a $3$-coloured complete graph whose underlying set is of the form $[k] = \{ 1, 2, \ldots, k \}$ for some
   $k$, where we write $\vert \sigma \vert = k$. A {\em $\sigma$-flag} is a $3$-coloured complete graph which contains a labelled
   copy of $\sigma$, or more formally a pair $(M, \theta)$ where $M$ is a $3$-coloured complete graph and
   $\theta$ is an injective map from $[k]$ to $V(M)$ 
   that respects the edge-colouring of $\sigma$.    Two $\sigma$-flags are {\em isomorphic} if there is an isomorphism
   that respects the labelling. More formally, $f$ is a {\em flag isomorphism} from  $(M_1, \theta_1)$ to $(M_2, \theta_2)$
   if $f: V(M_1) \rightarrow V(M_2)$ is an isomorphism of coloured graphs and $\sigma \circ \theta_1 = \theta_2$.

   We denote by ${\mathcal F}^\sigma_l$ the set of isomorphism classes of $\sigma$-flags with $l$ vertices. Note that
   if $0$ is the empty type then ${\mathcal F}^0_l = {\mathcal M}_l$. The flags of most interest to us are the elements
   of ${\mathcal F}^\sigma_4$ for various $\sigma$ with $\vert \sigma \vert = 3$; it is easy to see that
   if $\vert \sigma \vert = 3$ then  $\vert {\mathcal F}^\sigma_4 \vert = 27$.

   The notion of graph density described in the preceding section extends to $\sigma$-flags in a straightforward way. 
   Given $\sigma$-flags $F \in {\mathcal F}^\sigma_l$ and $G \in {\mathcal F}^\sigma_m$ for $m \ge l$, we define $p(F, G)$ to be the
   density of isomorphic copies of $F$ in $G$.  More formally let $G = (M, \theta)$,  choose uniformly at random a set
   ${\mathbf V} \in [V(M)]^l$ such that ${\mathbf V}$ contains $\im(\theta)$, and define $p(F, G)$ to be the probability
   that $(M[{\mathbf V}], \theta)$ is isomorphic (as a $\sigma$-flag) to $F$. By convention we will set $p(F, G) = 0$
   in case $m < l$.

   It is routine to see that if $l \le m \le n$, $F \in {\mathcal F}^\sigma_l$ and $H \in {\mathcal F}^\sigma_n$ 
then
\begin{equation} \label{chainrule}
  p(F, H) = \sum_{G \in {\mathcal F}^\sigma_m} p(F, G) p(G, H).
\end{equation} 
   This {\em chain rule} plays a central role in the theory.   
        
   More generally,  given flags $F_i  \in {\mathcal F}^\sigma_{l_i}$ for $1 \le i \le n$ and $G = (M, \theta)  \in {\mathcal F}^\sigma_m$ 
   where $m \ge \sum_i l_i - (n-1) \vert \sigma \vert$, we define a ``joint density'' $p(F_1, \ldots, F_n; G)$. 
   This is the probability that if we choose an $n$-tuple $({\mathbf V}_1, \ldots, {\mathbf V}_n)$ of subsets of $V(M)$ uniformly at random, subject
   to the conditions ${\mathbf V}_i \in [V(M)]^{l_i}$ and ${\mathbf V}_i \cap {\mathbf V}_j = \im(\theta)$ for $i \neq j$,
   then $(M[{\mathbf V}_i], \theta)$ is isomorphic to $F_i$
   for all $i$. 

   A sequence $(G_n)$ of $\sigma$-flags is said to be {\em increasing} if the number of vertices in $G_n$ tends to infinity,
   and {\em convergent} if the sequence of densities $(p(F, G_n))$  converges for every $\sigma$-flag $F$. A routine argument along the
   lines of the Bolzano-Weierstrass theorem shows that every increasing sequence has a convergent subsequence. If $(G_n)$ is convergent
   then we can define a map $\Phi$ on $\sigma$-flags by setting $\Phi(F) =\lim_{n \rightarrow \infty} p(F, G_n)$. We note that 
   when $F \in {\mathcal F}^\sigma_l$ and $l \le m$,  
   it follows readily from equation (\ref{chainrule}) that 
\begin{equation} \label{Phieqn1}
  \Phi(F) = \sum_{G \in {\mathcal F}^\sigma_m} p(F, G) \Phi(G).
\end{equation} 

    Equation (\ref{Phieqn1}) suggests that in some sense ``$F = \sum_{G \in {\mathcal F}^\sigma_m} p(F, G) G$'', and the
    definition of the {\em flag algebra} ${\mathcal A}^\sigma$ makes this precise. We define 
    ${\mathcal F}^\sigma_\infty = \bigcup_n {\mathcal F}^\sigma_n$, let ${\mathbb R}  {\mathcal F}^\sigma_\infty$ be the real
    vector space consisting of finite formal linear combinations of elements of  ${\mathcal F}^\sigma_\infty$, and then
    define ${\mathcal A}^\sigma$ to be the quotient of ${\mathbb R}  {\mathcal F}^\sigma_\infty$ by the subspace ${\mathcal K}^\sigma$ generated by
    all elements of the form  $F - \sum_{G \in {\mathcal F}^\sigma_m} p(F, G) G$. We will not be distinguishing
    between a flag $F$, its isomorphism class $[F] \in {\mathcal F}^\sigma$, the element $1 [F] \in {\mathbb R}  {\mathcal F}^\sigma_\infty$ 
    and the element $1 [F] +{\mathcal K}^\sigma \in {\mathcal A}^\sigma$.   

    If $\Phi$ is the map on $\sigma$-flags induced by a convergent sequence as above, then $\Phi$ extends by linearity to a 
    map $\Phi: {\mathbb R}  {\mathcal F}^\sigma_\infty \rightarrow {\mathbb R}$. The linear map $\Phi$ vanishes on ${\mathcal K}^\sigma$
    by equation (\ref{Phieqn1}), and hence induces a linear map $\Phi: {\mathcal A}^\sigma \rightarrow {\mathbb R}$.  
    So far ${\mathcal A}^\sigma$ is only a real vector space; we make it into an $\mathbb R$-algebra by defining
    a product as follows. Let $F_1 \in {\mathcal F}^\sigma_{l_1}$, $F_2 \in {\mathcal F}^\sigma_{l_2}$, let
    $m \ge l_1 + l_2 - \vert \sigma \vert$, and define
\[
    F_1 \cdot F_2 := \sum_{G \in {\mathcal F}^\sigma_m} p(F_1, F_2; G) G.
\]
    This can be shown \cite[Lemma 2.4]{FA} to give a well-defined multiplication operation on ${\mathcal A}^\sigma$ independent of the choice
    of $m$, and it can also be shown \cite[Theorem 3.3 part a]{FA} that if $\Phi: {\mathcal A}^\sigma \rightarrow {\mathbb R}$ is induced by a convergent
    sequence then $\Phi(F_1 \cdot F_2) = \Phi(F_1) \Phi(F_2)$, that is $\Phi$ is an algebra homomorphism from ${\mathcal A}^\sigma$ to ${\mathbb R}$.
    The converse is also true \cite[Theorem 3.3 part b]{FA}: if $\Phi$ is such a homomorphism and $\Phi(F) \ge 0$ for all $\sigma$-flags $F$, then
    there exists an increasing and convergent sequence $(G_n)$ such that $\Phi(F) = \lim_n p(F, G_n)$ for all flags $F$.

    Following Razborov we let $\mathrm{Hom}^+({\mathcal A}^\sigma, {\mathbb R})$ be the set of homomorphisms induced by
    convergent sequences of $\sigma$-flags, and define a preordering on ${\mathcal A}^\sigma$ by stipulating that
    $A \le B$ if and only if $\Phi(A) \le \Phi(B)$ for all $\Phi \in \mathrm{Hom}^+({\mathcal A}^\sigma, {\mathbb R})$.

\subsection{Averaging and lower bounds} \label{averaging}

    The algebra ${\mathcal A}^\sigma$ has an identity element $1_\sigma = (\sigma, id_\sigma)$, and it is easy to
    see that $\Phi(1^\sigma) = 1$ for all $\Phi \in \mathrm{Hom}^+({\mathcal A}^\sigma, {\mathbb R})$. Accordingly we will
    identify the real number $r$ and the element $r 1_\sigma$. With this convention, the task of finding
    asymptotic lower bounds for quantities like the density of monochromatic triangles amounts
    to proving inequalities of the form ``$F \ge r$ in ${\mathcal A}^0$'' for some sum of $0$-flags $F$ and real number $r$. 
    We will prove that
\[
    K^3_{red} + K^3_{blue} + K^3_{green} \ge 0.04.
\]

    Given a $\sigma$-flag $F  = (M, \theta)$, we let $F \vert_0 = M$. We define
    $\llbracket F \rrbracket_\sigma = q_\sigma(F) M$, where $q_\sigma(F)$ is the
    probability that a random injective function ${\boldsymbol \theta}$ from $[\vert \sigma \vert]$ 
    to $V(M)$ gives a $\sigma$-flag $(M, {\boldsymbol \theta})$ and this flag  is isomorphic to $F$.  
    This map on $\sigma$-flags  extends to a linear map from ${\mathcal A}^\sigma$ to
    ${\mathcal A}^0$. 

    A key fact is that for every type $\sigma$ and every $A \in {\mathcal A}^\sigma$, we have the inequality 
\begin{equation} \label{CS}
    \llbracket A^2 \rrbracket_\sigma \ge 0,
\end{equation}
    where $A^2 = A \cdot A$. 
    We will ultimately prove our desired
    lower bound by adding many inequalities of this form for various types $\sigma$ and elements $A$ of
    ${\mathcal A}^\sigma$.

   Inequality (\ref{CS}) can be proved by elementary means; roughly speaking 
   we average the square of the number of copies of $F$ containing a particular copy of $\sigma$  over all
   such copies and discard terms of low order. It can also be proved \cite[Theorem 3.14]{FA} using the 
   notion of  {\em random homomorphism} discussed below in subsection \ref{semidefinite}.

    We will prove that $K^3_{red} + K^3_{blue} + K^3_{green} \ge 0.04$ by proving an equation 
    of the form
\begin{equation} \label{maineqn}
    K^3_{red} + K^3_{blue} + K^3_{green} - 0.04 - \sum_i \llbracket L_i^2 \rrbracket_{\sigma_i} = \sum \lambda_k M_k,
\end{equation} 
    where the $\sigma_i$'s are types, $L_i \in {\mathcal A}^{\sigma_i}$, the  $M_k$'s are $3$-coloured
    complete graphs and  $\lambda_k \ge 0$ for all $k$.   
    Equation (\ref{maineqn}) clearly implies that $K^3_{red} + K^3_{blue} + K^3_{green} \ge 0.04$, which is the 
    translation of claim \ref{densityclaim} in Proposition \ref{prop:flag} into the flag language.
 
    Since there are increasing sequences of $3$-coloured complete graphs in which the density of monochromatic
    triangles approaches $0.04$, there are $\Phi \in \mathrm{Hom}^+({\mathcal A}^0, {\mathbb R})$ such that $\Phi(K^3_{red} + K^3_{blue} + K^3_{green}) =  0.04$. 
    For any such $\Phi$ we must have
\begin{enumerate}[(i)]
\item  $\Phi(\llbracket L_i^2 \rrbracket_{\sigma_i}) = 0$ for all $i$,
\item  $\lambda_k = 0$ for all $k$ such that $\Phi(M_k) > 0$, and
\item  $\Phi(M_k) = 0$ for all $k$ such that $\lambda_k > 0$. 
\end{enumerate} 
    The last of these points is the key to proving the second claim in Proposition \ref{prop:flag}. 
    We will verify that for all  $H \in {\mathcal H}$,
    $H$ is a linear combination of $M_k$'s such that $\lambda_k > 0$.  It follows that for all such $H$, 
    $\Phi(H) = 0$ for any $\Phi$ with $\Phi(K^3_{red} + K^3_{blue} + K^3_{green}) =  0.04$.
    This assertion is exactly the translation into flag language of  
    claim \ref{forbidclaim} in Proposition \ref{prop:flag}.  

\subsection{Proof of Proposition \ref{prop:flag}}

    To prove  Proposition \ref{prop:flag}  we need to specify  ten types, several hundred flags, and ten $27 \times 27$ matrices. 
    Rather than attempting to render the details of the proof in print, we have chosen to describe its structure 
    here and make all the data available online, together with programs which can be used to verify them.

    Let $\sigma$ be a type and let $L_1, \ldots, L_t \in {\mathcal A}^\sigma$, where each $L_i$
    is a real linear combination of a fixed set of $\sigma$-flags $F_1, \ldots, F_n$. By standard
    facts in linear algebra,
\[
    L_1^2 + \ldots + L_t^2 = \sum_{i j} Q_{i j} F_i \cdot F_j
\]
    for some $n \times n$ positive semidefinite symmetric matrix $Q$, and conversely any expression of 
    the form  $\sum_{i j} Q_{i j} F_i \cdot F_j$ for a positive semidefinite $Q$ is a sum of squares. 

    In our case we will have ten types $\tau_r$ for $1 \le r \le 10$, each with $\vert \tau_r \vert = 3$.
    The types are chosen to include representative elements of each isomorphism class of $3$-coloured
    triangles.  

    For each type $\tau_r$ we will have a complete list $F^r_1, \ldots, F^r_{27}$ of the $\tau_r$-flags
    on $4$ vertices. 
    In line with the discussion in subsection \ref{averaging}, we will specify for each $r$ an $27 \times 27$ symmetric
    matrix $Q^r$ and will actually prove an equation of the form 
\begin{equation} \label{maineqn2}
    K^3_{red} + K^3_{blue} + K^3_{green} - 0.04 - \sum_r \llbracket Q_r \rrbracket_{\sigma_r} = \sum \lambda_k M_k,
\end{equation} 
    where $Q_r = \sum_{i j} Q^r_{i j} F^r_i F^r_j$, each matrix $Q^r$ is positive semidefinite, and each coefficient
    $\lambda_k$ is non-negative.  
    The matrices $Q^r$ will have rational entries, so the whole computation can be done exactly using rational
    arithmetic.

    By the definition of flag multiplication, each product $F^r_i \cdot F^r_j$ can be written as a linear
    combination of elements of ${\mathcal F}^{\sigma_r}_5$, so each term $\sum_r \llbracket Q_r \rrbracket_{\sigma_r}$
    is a linear combination of elements of ${\mathcal M}_5$. The $3$-coloured complete graphs $M_k$ appearing in equation (\ref{maineqn2}) 
    will be the $792$ elements of ${\mathcal M}_5$. 
    Below we write ``$q(A, B)$'' for ``the coefficient of $B$ in the expansion of $A$''.

    Given the coefficient matrices $Q^r$, we must first verify that they are positive semidefinite (a routine calculation).
    We must then expand the left hand side of equation (\ref{maineqn2}) in the form $\sum \lambda_k M_k$, and check that
    $\lambda_k \ge 0$ for all $k$. Clearly
\[
    \lambda_k = p(K^3_{red}, M_k) + p(K^3_{blue}, M_k)  + p(K^3_{green},M_k)  - 0.04 - \sum_r q(\llbracket Q_r \rrbracket_{\sigma_r}, M_k),
\]
    and 
\[
   q(\llbracket Q_r \rrbracket_{\sigma_r}, M_k)  = \sum_{i j}  Q^r_{i j} q(\llbracket F^r_i \cdot F^r_j \rrbracket_{\sigma_r}, M_k),
\]
    so the main computational task in verifying the proof is to compute the coefficients $q(\llbracket F^r_i \cdot F^r_j \rrbracket_{\sigma_r}, M_k)$.

    A useful lemma of Razborov  gives a probabilistic interpretation of $q(\llbracket F^r_i \cdot F^r_j \rrbracket_{\sigma_r}, M_k)$
    which obviates the need to compute  $F^r_i \cdot F^r_j$ and $\llbracket F^r_i \cdot F^r_j \rrbracket_{\sigma_r}$
    before computing $q(\llbracket F^r_i \cdot F^r_j \rrbracket_{\sigma_r}, M_k)$. The lemma states that for any type $\tau$,
    any $\tau$-flags $K_1$ and $K_2$ and any $m$ which is large enough to express $\llbracket K_1 \cdot K_2 \rrbracket_\tau$
    as a linear combination of elements of ${\mathcal M}_m$, the coefficient $q(\llbracket K_1 \cdot K_2 \rrbracket_\tau, L)$
    of $L \in {\mathcal M}_m$  is the probability that
    choosing a random injection ${\boldsymbol \theta}$
    from $V(\tau)$ to $V(L)$ and then random sets ${\mathbf X}$ and ${\mathbf Y}$ of the appropriate size
    with ${\mathbf X} \cap {\mathbf Y} = \im({\boldsymbol \theta})$ gives flags $(L[{\mathbf X}], \theta)$ and $(L[{\mathbf Y}], \theta)$ such that
    $(L[{\mathbf X}], \theta)$ is isomorphic to $K_1$  and $(L[{\mathbf Y}], \theta)$ is isomorphic to $K_2$. 
    The proof is straightforward.    

   To complete the proof of Proposition \ref{prop:flag},  we must now compute the coefficients $\lambda_k$
 and verify that for all $k$
\begin{enumerate}[(i)]
\item $\lambda_k \ge 0$;
\item  For all $H \in {\mathcal H}$, 
 if $p(H, M_k) > 0$ then $\lambda_k > 0$. 
\end{enumerate}        
  The data for the proof and a Maple worksheet which verifies it can be found online at the URL
    {\tt http://www.math.cmu.edu/users/jcumming/ckpsty}. Further, the version of this paper on the
    arxiv has an appendix   with the data for the proof.

\subsection{Semidefinite programming}  \label{semidefinite}
 
    The proof described in the preceding section was obtained using {\em semidefinite programming}. 
    In our case we fixed the types $\tau_r$ and flags $F^r_j$, and set up a semidefinite programming problem
    where the unknowns are the  matrices $(Q^1, \ldots, Q^{10})$ and the goal is to maximise a lower bound for 
    $K^3_{red} + K^3_{blue} + K^3_{green}$.  Using the CSDP and SDPA solvers,
    we produced a proof of a lower bound of the form $0.04 - \epsilon$ where $\epsilon$ is about $10^{-6}$. 

    We now needed to perturb the coefficients in our matrices $Q^r$ to achieve the optimal value $0.04$
    for the lower bound. This was not completely trivial, because (as we already mentioned at the end
    of subsection \ref{averaging}) there are many constraints that must be satisfied by any choice
    of $Q^r$'s that achieves the optimal bound. Some of them are related to so-called
    {\em random homomorphisms} from~\cite[Section 3.2]{FA} as explained in~\cite[Section 4]{RazTuran}.
    If $\Phi \in \mathrm{Hom}^+({\mathcal A}^0, {\mathbb R})$ and $\sigma$ is a type
    such that (viewing $\sigma$ as an element of ${\mathcal A}^0$) $\Phi(\sigma) > 0$, then we may use $\Phi$ to  construct a certain probability measure on
    $\mathrm{Hom}^+({\mathcal A}^\sigma, {\mathbb R})$, which we may view (using  probabilistic language) as a {\em random homomorphism}
    ${\mathbf \Phi}^\sigma$.
    One of the properties of  ${\mathbf \Phi}^\sigma$ is that for any $F \in {\mathcal A}^\sigma$ the
    expected value of ${\mathbf \Phi}^\sigma(F)$ is given by the formula
\[
    E({\mathbf \Phi}^\sigma(F)) = \frac{\Phi(\llbracket F \rrbracket_\sigma)}{\Phi(\llbracket 1 \rrbracket_\sigma)}.
\] 
    So, we can view the inequality $[F^2]_\sigma \ge 0$ as an averaging argument
    analogous to the Cauchy-Schwartz theorem \cite[Theorem 3.14]{FA}.     
    
    Let $Q_r$ be one of the quadratic forms appearing in a proof of the optimal bound and
    let $Q_r=\sum_k m^2_k$ where each $m_k$ is a linear combination of the flags $F^r_i$.
    Recall from subsection \ref{averaging}
    that if $\Phi$  is such that  $\Phi(K^3_{red} + K^3_{blue} + K^3_{green}) =  0.04$,
    then $\Phi(\llbracket Q_r \rrbracket_{\tau_r}) = 0$ for all $r$.
    If $\Phi(\llbracket Q_r \rrbracket_{\tau_r})=0$, then it holds with probability one that
    ${\mathbf \Phi}^{\tau_r}(m_k) = 0$. This yields that all eigenvectors of $Q_r$
    corresponding to non-zero eigenvalues must lie in a certain linear space.

    However,
    in our case we could not derive enough relations of this kind from the known extremal $\Phi$. At this point we inspected our 			  	   	numerical data, in particular we analysed the eigenvectors
 		corresponding to the very small eigenvalues and we guessed additional relations to complete the
		proof. Oleg Pikhurko~\cite{pikcom} offered us the following explanation of the origin of these relations.
	 	It is possible to alter the known extremal $\Phi$ in such a way that some values of
	  ${\mathbf \Phi}^{\tau_r}$ change $\Theta(\eps)$ but the density of monochromatic triangles
	  changes only $O(\eps^3)$. This yields that ${\mathbf \Phi}^{\tau_r}(m_k) = O(\eps^{3/2})$
 		with probability one which further restricts the linear space containing all eigenvectors of
 		$Q_r$ corresponding to non-zero eigenvalues.

\section{Proof of Theorem~\ref{mainthm}}
\subsection{Finding a standard subgraph of $G$}
Define constants $\eps , \eps _1, \eps _2 , \eps _3, \eps _4, \eps _5, \eps _6, \eps _7 , \eps _8 , \eps _9 , \eps _{10}$ 
and integers $n_0, n_1, n_2$ such that $n_0$ and $\eps$ satisfy the assertion of Proposition~\ref{prop:flag} and
\begin{align}\label{hier}
0<1/n_0 \ll \eps \ll 1/n_1 \ll \eps _1 \ll \eps _2  \ll \eps _3 \ll 1/n_2 \ll \eps _4 \ll \eps _5 \ll \eps _6 \ll
\eps _7 \ll \eps _8 \ll \eps _9 \ll \eps _{10} \ll 1.
\end{align}
 Let $G$ be a $3$-coloured complete graph on $n \geq n_0$ vertices with  
$p({\mathcal K}^3, G)$ minimised. We may assume the three colours used are red, green and blue. 
Note that, by the minimality of $G$,
 $p({\mathcal K}^3, G)  \leq p({\mathcal K}^3, G_{ex} (n)) \leq 0.04$.
  Since $n\ge n_0$, Proposition~\ref{prop:flag} implies that
  $p({\mathcal K}^3, G)  \geq 0.04 - \eps$ and $p({\mathcal H}, G) < \eps$.

Let us call an induced subgraph $G' \subseteq G$ {\em $\eps_1$-standard} if 
\begin{itemize}
 \item[(i)] $p({\mathcal K}^3, G') \le 0.04+\eps_1$;
 \item[(ii)] $p({\mathcal H}, G') =0$.
\end{itemize}

Now we randomly pick $n_1$ vertices from $G$ to induce a subgraph $G'$.
\begin{claim}\label{cstan}
$\Pr(G'\mbox{ is $\eps_1$-standard}) \geq 1-\eps_2$.
\end{claim}
\proof
Since $1 /n_1 \ll \eps _1$,
Proposition~\ref{prop:flag} implies that $p^{\rm min}({\mathcal K}^3, n_1) > 0.04-\eps ^2 _1$. Thus, $Z:=p({\mathcal K}^3, G')-(0.04 -\eps ^2 _1) >0$.
Note that $\mathbb E (Z) \leq \eps _1 ^2$ since $p({\mathcal K}^3, G) \leq 0.04$. Hence, by Markov's inequality,
$$ \Pr(Z\ge \eps _1)\leq \frac{\eps _1 ^2}{\eps _1}= \eps_1$$
and therefore
$$ \Pr (p({\mathcal K}^3, G') \leq 0.04+\eps _1 )\geq 1-\eps _1.$$
By Markov's inequality,
\[
\Pr\left(   p({\mathcal H}, G')   <\frac{ 2\eps}{\eps_2}\right)\geq  1-\eps_2 /2.
\]
Note that (\ref{hier}) implies that  $2 \eps \binom{n_1}{ 4} / \eps _2 <1$. Thus, the claim follows.
\endproof
In the next two subsections we will build up structure in our $\eps_1$-standard subgraphs $G'$, 
thereby obtaining that each such $G'$ has
`similar' structure to $G_{ex} (n_1)$.

\subsection{Properties of maximal monochromatic cliques in $G'$}
Consider any $\eps _1$-standard subgraph $G'$ of $G$ on $n_1$ vertices.
Let $\mathcal X$ be the set of maximal monochromatic cliques of order at least $4$ in $G'$.
So a clique $X_1$ in $\mathcal X$ cannot strictly contain another clique $X_2 \in \mathcal X$. However,
$\mathcal X$ may contain cliques that intersect each other. 
Since $n_1$ is sufficiently large, $G'$ contains a ${\mathcal K}^4$ by Ramsey's theorem. Thus, $|\mathcal X|\ge 1$.

\begin{claim}\label{c:mono2}
Let $X\in \mathcal X$ and $y \in V(G') \backslash V(X)$. All but one of the edges $xy$ with $x\in V(X)$ have the same colour, which is different from the  colour of $X$. 
The remaining edge is either of that same colour or of the colour of $X$.
\end{claim}
\proof
Assume $X$ is coloured red. By definition of $\mathcal X$, we cannot have that all edges between $X$ and $y$ are red. This implies
that at most one such edge is red (else $G'$ contains an  ${\mathcal H}(2,1,0)$, a contradiction to (ii)). This in turn implies
that there does not exist both green and blue edges between $X$ and $y$ (else $G'$ contains an  ${\mathcal H}(0,2,1)$). The claim
now follows.
\endproof
\medskip

\begin{claim}\label{noint}
Suppose $X_1, X_2 \in \mathcal X$ have different colours. Then $X_1$ and $X_2$ are vertex-disjoint.
\end{claim}
\proof
Since $X_1$ and $X_2$ have different colours, $|V(X_1)\cap V(X_2)|\leq 1$. Suppose for a contradiction there exists
a vertex $x \in V(X_1) \cap V(X_2)$. Suppose $X_1$ is red and $X_2$ is blue. 
For each $x_1 \in X_1 - x$, since $x_1 x$ is red, Claim~\ref{c:mono2} implies that all but at most one of the edges from $x_1$
to $X_2$ are red. Thus, there exists distinct $x',x'' \in X_1 -x$ and $x''' \in X_2 -x$ such that
$x'x'''$ and $x''x'''$ are red. But since $xx'''$ is blue, $G'[x,x',x'',x''']$ is an ${\mathcal H}(2,1,0)$, a contradiction
to (ii).
\endproof
\medskip

\begin{claim}~\label{c:mono3}
 \begin{itemize}
 \item[(a)] If $X_1,X_2\in \mathcal X$ have different colours, then there is a vertex $v_1\in V(X_1)$ and a vertex 
 $v_2\in V( X_2)$ 
 such that all edges between $X_1-v_1$ and $X_2-v_2$ have the same colour, and this colour is different from the colours of $X_1$ and $X_2$.
 \item[(b)] If $X_1,X_2\in \mathcal X$ have the same colour, then either $X_1$ and $X_2$ share exactly one vertex $v$, and  all edges  between $X_1-v$ and $X_2-v$ have a common colour, or $X_1$ and $X_2$ are disjoint, there is a (possibly empty) matching of the colour of $X_1$ and $X_2$ between $X_1$ and $X_2$, and  all other edges between $X_1$ and $X_2$ have the same colour, different from the colour of $X_1$ and $X_2$.
 \end{itemize}
\end{claim}
\proof If $X_1, X_2 \in \mathcal X$ have different colours, then by Claim~\ref{noint}, $X_1$ and $X_2$ are
vertex-disjoint. Suppose $X_1$ is red and $X_2$ is blue. Firstly, note that there does not exist distinct
$x'_1,x''_1 \in X_1$ and $x'_2, x''_2 \in X_2$ such that both $x'_1x'_2$ and $x''_1x''_2$ are blue.
Indeed, if such edges exist then by Claim~\ref{c:mono2}, $x'_1x''_2$ and $x''_1x'_2$ are red. Again by Claim~\ref{c:mono2},
this implies that every edge from $x'_2 $ to $X_1-x''_1$ is blue and every edge from $x''_2 $ to $X_1-x'_1$ is blue.
Let $a,b \in X_1 -\{x'_1,x''_1\}$. Then $G'[a,b,x'_2, x''_2]$ is an ${\mathcal H}(2,1,0)$, a contradiction.

An identical argument implies that there does not exist distinct
$x'_1,x''_1 \in X_1$ and $x'_2, x''_2 \in X_2$ such that both $x'_1x'_2$ and $x''_1x''_2$ are red.
By Claim~\ref{c:mono2} this implies that there exists at most one vertex $v_1 \in X_1$ such that $v_1$ sends at least
one red edge to $X_2$  and there exists at most one vertex $v_2 \in X_2$ such that $v_2$ sends at least
one blue edge to $X_1$. This implies that all the edges from $X_1-v_1$ to $X_2-v_2$ are green, and so (a) is satisfied.

Suppose $X_1,X_2 \in \mathcal X$ have the same colour, red say. Notice that $|V(X_1)\cap V(X_2)|\leq 1$, since
otherwise a vertex  in $V(X_1)\backslash V(X_2)$ would send at least two red edges to $X_2$, a contradiction to 
Claim~\ref{c:mono2}. If $|V(X_1)\cap V(X_2)|= 1$ then it is easy to see that, by Claim~\ref{c:mono2},
the first part of (b) holds.

If $X_1$ and $X_2$ are disjoint, then by Claim~\ref{c:mono2}, no vertex in $X_1$ sends more than one red edge to $X_2$ and
no vertex in $X_2$ sends more than one red edge to $X_1$. Thus, the red edges between $X_1$ and $X_2$ form a (possibly empty)
matching. Applying Claim~\ref{c:mono2} again shows that the second part of (b) holds.
\endproof

\subsection{Properties of the clique graph}
We now define a new $3$-coloured complete graph $F$ which we refer to as the \emph{clique graph}.
The vertex set of $F$ consists of the elements of $\mathcal X$ together with the vertices in $Y$
where $Y\subseteq V(G')$ is the set of vertices in $G'$ not contained in any of the cliques in $\mathcal X$. 
If $x,y \in Y$ then, in $F$, we colour $xy$ with the colour of $xy$ in $G$. 
If $X_1, X_2 \in \mathcal X$ then, in $F$, we colour the edge $X_1X_2$ with the colour of the majority
of the edges between $X_1$ and $X_2$ in $G$. (Note that this colour is well-defined by Claim~\ref{c:mono3}.)
Finally, given a vertex $y \in Y$ and $X \in \mathcal X$, in $F$ we colour the edge $yX$ with the colour of the majority
of the edges between $y$ and $X$ in $G$. (This colour is well-defined by Claim~\ref{c:mono2}.)

\begin{claim}\label{noX}
 No  ${\mathcal K}^3$ in $F$ contains a vertex $X \in \mathcal X$. Moreover, $F$ contains no ${\mathcal K}^4$.
\end{claim}
\proof
The first part of the claim follows from Claim~\ref{c:mono2} since otherwise there would be an  ${\mathcal H}(2,1,0)$
 in $G'$, a contradiction to (ii). The second part of the claim follows from the first part together with the definition of $Y$.
\endproof

For every clique $X\in \mathcal X$, the edges in $F$ leaving $X$ must have different colours from $X$. Thus, 
 we have $|\mathcal X|+|Y|\le 35$. Indeed,
  otherwise each $X$ in $\mathcal X$ is incident to $18$ edges of the same colour in $F$. But then,
 since $R(4,4)=18$, $F$ contains a    ${\mathcal K}^4$ or a  ${\mathcal K}^3$ containing $X$, a contradiction 
 to Claim~\ref{noX}. If $|\mathcal X|\le 4$, then
  $p({\mathcal K}^3, G') \ge 4{\lfloor(n_1-34)/4 \rfloor\choose 3}/{n_1\choose 3}>0.04+\eps_1$, a contradiction to (i). Thus, $|\mathcal X|\geq 5$.
 
If there are three cliques in $\mathcal X$ of one colour, and another clique in $\mathcal X$
 of a different colour, then it is easy to see by Claim~\ref{c:mono3} that there must be a monochromatic triangle between these four cliques, a contradiction to Claim~\ref{noX}. Similarly, we cannot have two cliques in $\mathcal X$ of one colour, and also cliques in 
 $\mathcal X$ of the other two colours. Therefore, all cliques in $\mathcal X$ must have the same colour, say red. 
 
If $|\mathcal X|\geq 6$, then again $F[\mathcal X]$ contains a ${\mathcal K}^3$, a contradiction. So $|\mathcal X|=5$. 
Further, $Y =\emptyset$, since otherwise $F[\mathcal X \cup \{y\}]$ is $2$-coloured and thus contains a ${\mathcal K}^3$ (for all $y \in Y$). 

\begin{claim}\label{properties}
Let $\mathcal X= \{X_1, \dots , X_5\}$. The following properties hold:
\begin{itemize}
\item[$(\alpha _1)$] $(1-\eps_3)\frac{n_1}{5} \leq |X_i|\leq (1+\eps_3)\frac{n_1}{5}$ for all $1 \leq i \leq 5$;
\item[$(\alpha _2)$] $E(F)$ is $2$-coloured with green and blue and consists of a green $5$-cycle and a blue $5$-cycle. We may assume that $X_1X_2X_3X_4X_5X_1$ is a green cycle and $X_1X_3X_5X_2X_4X_1$ is a blue cycle;
\item[$(\alpha _3)$] Either the cliques in $\mathcal X$ are vertex-disjoint or there exists a unique vertex $w$ that lies
in each clique in $\mathcal X$ (and $w$ is the only vertex which lies is more than one clique in $\mathcal X$). 
\end{itemize}
\end{claim}
\proof
Every clique in $\mathcal X$ contains at least $(1-\eps_3)\frac{n_1}{5}$ vertices as otherwise
\[
 p({\mathcal K}^3, G') \ge \left({(1-\eps_3)\frac{n_1}{5}\choose 3}+4{(1+\eps_3/4)\frac{n_1}{5}\choose 3}\right)\left/{n_1\choose 3}\right.
 \stackrel{(\ref{hier})}{>} 0.04+\eps_1.
\]
A similar calculation shows that every clique in $\mathcal X$ contains at most $(1+\eps_3)\frac{n_1}{5}$ vertices.
Every clique in $\mathcal X$ is red, thus $E(F)$ is $2$-coloured with green and blue. Since $F$ does not contain a monochromatic
triangle, $F$ must satisfy ($\alpha _2$). 

Suppose two of the cliques, say $X_1$ and $X_2$, share a vertex $w$. As $X_1X_3$ is blue and $X_2X_3$ is green, 
Claim~\ref{c:mono2} implies that,
for every vertex $v \in X_3$ the edge $vw\in E(G)$ can be neither blue nor green, so it has to be red. But this implies that $w\in X_3$. By similar arguments, $w\in X_4\cap X_5$. Thus, ($\alpha _3$) holds.
\endproof

\subsection{Obtaining structure in $G$ from $G'$} 
Our next task is to find a special set $V' \subseteq V(G)$ such that $G[V']$ has `similar' structure to $G_{ex} (n_2)$. 
\begin{claim}\label{cV'}
There exists a set $V' \subseteq V(G)$ such that the following properties hold:
\begin{itemize}
\item[$(\beta _1)$] $|V'|=n_2$;
\item[$(\beta _2)$] $V'$ has a partition into non-empty sets $C_1,C_2,C_3,C_4,C_5$ such that
\begin{itemize}
\item[$\bullet$] $\frac{|C_i|}{|C_j|}>1-\eps_4$ for all $1 \le i\ne  j \le 5 $,
\item[$\bullet$] all edges inside the $C_i$ have the same colour, say red,
\item[$\bullet$] all edges between $C_i$ and $C_{i+1}$ are green, 
\item[$\bullet$] all edges between $C_i$ and $C_{i+2}$ are blue (here indices are computed modulo $5$);
\end{itemize}
\item[$(\beta _3)$] If we uniformly at random choose two vertices $u,v\in V(G)$, then with probability greater than $1-\eps_5$, the set $V'\cup\{u,v\}$ satisfies~($\beta _2$) as well.
\end{itemize}
\end{claim}
\proof
Consider any $\eps_1$-standard subgraph $G'$ of $G$ on $n_1$ vertices. Randomly select a set $W \subseteq V(G')$ of
size $n_2$. Then with probability more than $1-\eps  _4 ^3$, $W$ satisfies ($\beta _2$). This follows from 
Claims~\ref{c:mono3}~and~\ref{properties}. For example, by applying a Chernoff-type bound for the
hypergeometric distribution (see e.g.~\cite[Theorem 2.10]{Janson&Luczak&Rucinski00}),
($\alpha _1$) implies that with probability greater than $1- \eps _4 ^4 $, the first two conditions in
($\beta _2$) hold.
Further, note that the probability that $W$ contains the special vertex $w$ from ($\alpha _3$) (if it exists) is
$n_2 /n_1 \ll \eps _4$ by (\ref{hier}).
 
Randomly select a set $W' \subseteq V(G)$ of size $n_2$. One can view this procedure as first randomly selecting a set $W''
\subseteq V(G)$ of size $n_1$, then randomly selecting a set $W' \subseteq W''$ of size $n_2$. By Claim~\ref{cstan},
with probability at least $1- \eps_2$, $G[W'']$ is $\eps_1$-standard.

Together, this implies that with probability greater than $(1- \eps _2)(1-\eps _4^3)> 1- \eps  _4^2$ 
a randomly chosen set $W' \subseteq V(G)$ of size $n_2$ satisfies ($\beta _2$). Similarly, 
with probability greater than $ 1- \eps  _4 ^2$ 
a randomly chosen set $W' \subseteq V(G)$ of size $n_2+2$ satisfies ($\beta _2$).

Consider all pairs $(V', \{u,v\})$ such that $\{u,v\}, V' \subseteq V(G)$ and $|V'|=n_2$.
(Note here we allow for $V' \cap \{u,v\}\not = \emptyset$.) With probability greater than
$1-3\eps _4^2 $, a randomly selected such pair $(V', \{u,v\})$ has the property
that both $V'$ and $V' \cup \{u,v \}$ are $\eps_1$-standard. Since $3\eps _4^2 \ll \eps _5$, this implies that
there exists a set $V' \subseteq V(G)$ satisfying ($\beta_1$)--($\beta _3$).
\endproof

Let $V'$ be as in Claim~\ref{cV'}. Set
\[
 E_0:=\{uv\in E(G): V'\cup\{u,v\}\mbox{ does not satisfy~($\beta _2$)}\}.
\]
Then $|E_0|<\eps_5 n^2$ by ($\beta _3$). Let 
\[
 V_0:=\{v\in V(G): v\mbox{ is incident to at least $\eps_6 n$ edges in }E_0\}.
\]
Then $|V_0|<\eps_6 n$ since $\eps _5 \ll \eps _6$. For each $1 \leq i \leq 5$ define
\[
 F_i:=\{ v\in V(G)\setminus V_0:vw\mbox{ is red for all }w\in C_i\}.
\]
Note that $V(G)=V_0\cup F_1\cup F_2\cup F_3\cup F_4\cup F_5$.
Further, notice that the $F_i$ are disjoint. (Indeed, if there is a vertex $x \in F_i \cap F_j$ for some $i \not = j$ 
then all edges incident to $x$ are in $E_0$. But then $x \in V_0$, a contradiction.)
\begin{claim}\label{Fibound}
For all $1 \leq i \leq 5$,
\begin{align*}
(1-\eps _7)n/5 \leq |F_i| \leq (1+\eps _7)n/5.
\end{align*}
\end{claim}
\proof
Suppose $|F_i|< (1-\eps _7)n/5$ for some $1 \leq i \leq 5$. By definition of the $F_j$ and ($\beta _3$), there are at most
$\eps _5 n^2$ edges in $F_j$ that are not red (for each $1 \leq j \leq 5$). Thus, in each $F_j$, there are at most $\eps _5 n^3$ triples that do not form a red triangle. Hence, there are at least
\begin{align*}
\binom{|F_i|}{3} +4\binom{|V(G) \backslash (V_0 \cup F_i)|/4}{3} -5\eps_5 n^3 & \geq  \binom{(1-\eps _7)n/5}{3}+4 \binom{(4/5+\eps _7/5 -\eps _6)n/4}{3}-5\eps_5 n^3 \\ & \stackrel{(\ref{hier})}{>} (0.04+\eps_1)\binom{n}{3}
\end{align*}
red triangles in $G$, a contradiction. The upper bound  follows similarly.
\endproof

For each $v\in V(G)$ and $1\le i\le 5$, let $r_i(v):=|N_{red}(v)\cap F_i|$, $b_i(v):=|N_{blue}(v)\cap F_i|$ and
 $g_i(v):=|N_{green}(v)\cap F_i|$. On the basis of these quantities, we define another partition of $V(G)$ as follows. For each
$1 \leq i \leq 5$, set
\[
 V_i:=\left\{v\in V(G):\begin{array}{rl}
                        r_i(v)&\ge 0.199n,\\
                        g_{i+1}(v)&\ge 0.199n,\\
                        b_{i+2}(v)&\ge 0.199n,\\
                        b_{i+3}(v)&\ge 0.199n,\mbox{ and}\\
                        g_{i+4}(v)&\ge 0.199n
                       \end{array}
\right\}.
\]
\begin{claim}\label{csub}
For each $1 \leq i \leq 5$, $F_i \subseteq V_i$. 
\end{claim}
\proof
Given any $v \in F_i$, $v$ is incident to at most $\eps _6 n$ edges in $E_0$. Thus, there are at most $\eps _6 n$ vertices
in $F_i$ that $v$ does not send a red edge to. Hence, Claim~\ref{Fibound} implies that $r_i (v) \geq 0.199n$. Similar
arguments give $g_{i+1}(v), b_{i+2}(v),  b_{i+3}(v), g_{i+4}(v) \geq 0.199n$.
\endproof
Set $V^*:=V(G)\setminus (V_1\cup V_2\cup V_3\cup V_4\cup V_5)$. Let $E^*$ be the set of edges $xy$ in $G[V_1\cup V_2\cup V_3\cup V_4\cup V_5]$ such that $x \in V_i$ and $y \in V_j$ for some $1 \leq i , j \leq 5$ and so that the colour of $xy$ differs from that of the
edges between $C_i$ and $C_j$. 

\begin{claim}\label{c:Vred}
For each $1 \leq i \leq 5$, $G[V_i]$ is a red clique.
\end{claim}
\begin{proof}
Claims~\ref{Fibound} and~\ref{csub} imply that $(1-\eps_{7})n/5 \leq|V_i| \leq(1+4\eps_{7})n/5$ for all $1 \leq i \leq 5$.
 Suppose for a contradiction that there is a blue edge $vw$ with $v,w\in V_1$. Recolouring $vw$ red creates at most 
\[
 |V^*|+|V_1|+(0.004+\eps_{7}/5) n < 0.205n
\]
 new red triangles. (The $(0.004+\eps_{7}/5) n$ term counts the maximum number of red edges a vertex in $V_1$ can send to
 $V_2\cup V_3\cup V_4\cup V_5$.)
 On the other hand, the recolouring destroys at least 
 $$|F_3|+|F_4|-2(0.002+2\eps_{7}/5)n >0.395n$$
 blue triangles, contradicting the minimality of  $G$.
\end{proof}

\begin{claim}\label{E*}
$E^* \subseteq E_0$.
\end{claim}
\proof
Suppose $xy \in E^*$ where $x \in V_i$ and $y \in V_j$ for some $1 \leq i,j \leq5$. 
The colour of $xy$ differs from that of the edges between $C_i$ and $C_j$. 
But Claim~\ref{c:Vred} implies that $x$  only sends red edges to $C_i$ and $y$ only sends 
 red edges to $C_j$. Thus, $xy \in E_0$.
\endproof

\begin{claim}\label{c:V*}
 $V^*=\emptyset$.
\end{claim}
\begin{proof}
 Suppose that $v\in V^*$. We count the number of monochromatic triangles $t_v$ containing $v$ and two vertices from
  outside of $V^*$. 
First, if we were to recolour all edges from $v$ to the smallest $V_i$ red, from $v$ to $V_{i+1}\cup V_{i+4}$ green, and from $v$ to $V_{i+2}\cup V_{i+3}$ blue, then we would get at most 
\[
 \binom{|V_i|}{2}+|E^*| \leq \binom{\lfloor n/5 \rfloor}{2} +|E_0|< (0.02+\eps_{5})n^2
\]
monochromatic triangles containing $v$ and two vertices from outside of $V^*$, and at most $|V^*|n< \eps _6 n^2$ new triangles containing $v$ and another vertex from $V^*$. Thus, the minimality of $G$ implies that 
\begin{align}\label{up}
 t_v< (0.02+\eps_{5}+\eps_6)n^2.
\end{align}

Recall our notation $r_i(v),g_i(v),b_i(v)$. 
Note that
\begin{align}
 t_v  \ge & \ 0.5(r_1(v)^2+r_2(v) ^2+r_3(v) ^2+r_4(v)^2+r_5(v)^2) \label{lower} \\
&+g_1(v)g_2(v)+g_2(v)g_3(v)+g_3(v)g_4(v)+g_4(v)g_5(v)+g_5(v)g_1(v)\nonumber \\
&+b_1(v)b_3(v)+b_2(v)b_4(v)+b_3(v)b_5(v)+b_4(v)b_1(v)+b_5(v)b_2(v)\nonumber \\
&-2\eps_{5}n^2. \nonumber
\end{align}
where the last term occurs since $\binom{r_i(v)}{2} \geq 0.5r_i ^2-n$ for each $1 \leq i \leq 5$ and as $|E^*|< \eps _5 n^2$. 

Our next task is to find a lower bound on
\begin{align}
& 0.5(r_1(v)^2+r_2(v) ^2+r_3(v)^2+r_4(v)^2+r_5(v)^2)+
\gamma_1\gamma_2+\gamma_2\gamma_3+\gamma_3\gamma_4+\gamma_4\gamma_5+\gamma_5\gamma_1 \label{low}\\
 & +\beta_1\beta_3+\beta_2\beta_4+\beta_3\beta_5+\beta_4\beta_1+\beta_5\beta_2 \nonumber
\end{align}
 under the assumptions that $\gamma _i,\beta_i \geq 0$ are integers and
$|F_i|=r_i(v)+\gamma_i+\beta_i$  for all $1 \leq i \leq 5$.
(Note that finding a lower bound on (\ref{low}) gives us a lower bound on the right hand side of (\ref{lower}) and
thus a lower bound on the value of $t_v$.)
 Notice that there is a choice of the 
values of the $\gamma _i$ and $\beta _i$ which minimise the value of (\ref{low}) and which
satisfy $\gamma _i=0$ or $\beta _i=0$ for all $1\leq i \leq 5$. (For example,
if there is a choice of the 
values of the $\gamma _i$ and $\beta _i$ which minimise the value of (\ref{low})
but with $\gamma _1, \beta _1>0$ then this implies that $\gamma _2+\gamma _5=\beta _3+ \beta_4$. 
We can thus obtain another `minimal'  choice of the $\gamma _i $ and $\beta _i$ by resetting $\gamma _1=0$ and 
$\beta _1=|F_1|-r_1(v)$.)

Consider such a choice of the $ \gamma_i$ and $\beta _i$. So at least three of the $\gamma_i$ equal $0$ or at least three of the
$\beta_i$ equal $0$. Assume that $\beta_1=\beta_2=0$.
Thus,
\begin{align}\label{1l}
 0.5 r_1(v) ^2+0.5 r_2(v) ^2+\gamma_1\gamma_2\ge (0.02-\eps_{8})n^2
\end{align}
since $r_1(v)+\gamma_1, r_2(v)+\gamma_2 \geq (1-\eps _7)n/5$.
If $\gamma_3=\gamma_5=0$, then similarly
\[
 0.5 r_3 (v) ^2+0.5 r_5(v) ^2+\beta_3\beta_5\ge (0.02-\eps_{8})n^2.
\]
Together with (\ref{lower}) this implies that $t_v\ge (0.04-2\eps_{8}n^2-2\eps _5) n^2$, a contradiction to (\ref{up}).
So $\beta_3=0$ or $\beta_5=0$.
Assume that $\beta_3=0$. Thus, as before we have that
\begin{align}\label{2l}
 0.5 r_2(v) ^2+0.5 r_3(v) ^2+\gamma_2\gamma_3\ge (0.02-\eps_{8})n^2.
\end{align}
Hence, (\ref{1l}) and (\ref{2l}) imply that (\ref{low}) is bounded below by
$$(0.04-2\eps _8)n^2 - 0.5r_2 (v) ^2.$$

In all other cases we obtain that 
(\ref{low}) is bounded below by
$$(0.04-2\eps _8)n^2 - 0.5r_{j'} (v) ^2$$
for some $1 \leq j' \leq 5$.
In particular, together with (\ref{lower}) this implies that
$$t_v \geq (0.04-2\eps _8)n^2 - 0.5r_{j'} (v) ^2-2 \eps _5 n^2$$
for some $1 \leq j' \leq 5$. Thus, (\ref{up}) implies that $r_{j'} (v) \geq (0.2-\eps _9)n$ for some $1 \le j '\le 5$.
This in turn implies that $v$ lies in at least $\binom{(0.2-\eps _9)n}{2}-|E^*| \geq (0.02 -\eps _9)n^2$
red triangles in $G$.
Note that (\ref{up}) also implies that $r_i (v) < \eps _{10} n$ for all $i \in [5] \setminus \{j'\}$.

We may assume that $j'=1$. 
Suppose that for some $j$, $g_{j}(v)\ge 0.0001n$ and $b_{j} (v)\ge 0.0001n$. Let $\{i_1,i_2,i_3\}=[5]\setminus\{1,j\}$.
It is easy to see that this implies that there are at least
$$(0.0001n)^2-|E^*|$$
green or blue monochromatic triangles containing $v$ and vertices from $V_{j}$, $V_{i_1}$, $V_{i_2}$ and $V_{i_3}$.
Therefore, $t_v \geq (0.02-\eps _9)n^2+(0.0001n)^2-|E^*|$, a contradiction to (\ref{up}).

Thus, for every $i\in\{2,3,4,5 \}$, either $g_i(v)<0.0001n$ or $b_i(v)<0.0001n$. 
If $b_2(v) \geq 0.0001n$ then it is easy to see that $b_4 (v), b_5 (v) < 0.0001n$ (else we get $(0.0001n)^2-|E^*|$ blue triangles
containing $v$, a contradiction). So $g_4 (v), g_5 (v) \geq 0.19n$. This implies that there are at least
$(0.19n)^2-|E^*|$ green triangles containing $v$, a contradiction. Thus, $b_2(v) < 0.0001n$. Similar
arguments imply that $g_3(v), b_4(v), g_5 (v) <0.0001n$. This implies that $v \in V_1$, a contradiction. So indeed
 $V^* =\emptyset$, as desired.
\end{proof}
By Claims~\ref{c:Vred} and ~\ref{c:V*}, $V(G)$ can be partitioned into $5$ monochromatic cliques of the same colour.
A straightforward calculation yields that
the graphs in $\mathcal G_n$ are precisely
those $3$-coloured complete graphs on $n$ vertices
that minimise the number of monochromatic triangles
among all $3$-coloured complete graphs
whose vertex set can be partitioned into $5$ monochromatic cliques of the same colour.
Thus, $G\in \mathcal G_n$ as desired.

\section*{Acknowledgements}

  Thanks to Tom Bohman, Po-Shen Loh, John Mackey, Ryan Martin, Dhruv Mubayi, and Oleg Pikhurko for their interest
   and encouragement.   The computational part was done using  the Maple\footnote{Maple is a trademark of Waterloo Maple Incorporated.}
   computer algebra system \cite{Maple} and the freely available CSDP \cite{csdp}
  and SDPA \cite{sdpa} semidefinite programming solvers.
  In particular the robustness of CSDP and the availability of very high precision versions of SDPA were critical.

\section*{Appendix}

In this appendix we give the data for the proof of Proposition \ref{prop:flag}.

    We will describe the types,models and  flags which we  use by ``adjacency matrices''. Our convention is that
    the numbers $1$, $2$ and $3$ correspond to the colours red, blue and green respectively. If $\sigma$
    is a type of size $k$ then $\sigma$ is described by a symmetric $k \times k$ matrix in which
    the $(i,j)$ entry
    is the number corresponding to the colour of the edge $ij$ for $i \neq j$, and is zero for $i = j$.
    Similarly if $M$ is a model with $v(M) = n$ then we enumerate the vertices as $v_1, \ldots v_n$, and
    describe $M$ by an $n \times n$ matrix in which  the $(i,j)$ entry
    is the number corresponding to the colour of the edge $v_i v_j$ for $i \neq j$, and is zero for $i = j$.
    
    When $\sigma$ is a type of size $k$ and $F = (M, \theta)$ is a $\sigma$-flag then we can  enumerate
   the vertices of $M$ so that $\theta(i) = v_i$ for $1 \le i \le k$. It follows that the matrix of $M$
   contains the matrix of $\sigma$ in the first $k$ many rows and columns.   In particular when $k = 3$ and
   $n =4$, which is the only case of interest for us here, we may completely describe the $\sigma$-flag $F$ by
   specifying $\sigma$ and  a row vector $v$ of length $3$ containing the $(4,1)$, $(4,2)$ and $(4,3)$ entries 
   of $M$; we will denote the corresponding $\sigma$-flag as ``$v_\sigma$''.

   There are ten types of size $3$ up to isomorphism, all of which are used. For each type $\sigma_i$ we list
   the  $27$ $\sigma_i$-flags on $4$ vertices as $F^i_j$. We then list the ten matrices $Q^i$. 

\tiny

\[
 \sigma_1=\left[ \begin {array}{ccc} 0&1&1\\ \noalign{\medskip}1&0&1
\\ \noalign{\medskip}1&1&0\end {array} \right], \sigma_2= \left[ \begin {array}
{ccc} 0&1&1\\ \noalign{\medskip}1&0&2\\ \noalign{\medskip}1&2&0
\end {array} \right],  \sigma_3=\left[ \begin {array}{ccc} 0&1&1
\\ \noalign{\medskip}1&0&3\\ \noalign{\medskip}1&3&0\end {array}
 \right],  \sigma_4=\left[ \begin {array}{ccc} 0&1&2\\ \noalign{\medskip}1&0&2
\\ \noalign{\medskip}2&2&0\end {array} \right],
\]

\[
  \sigma_5= \left[ \begin {array}
{ccc} 0&1&2\\ \noalign{\medskip}1&0&3\\ \noalign{\medskip}2&3&0
\end {array} \right],   \sigma_6= \left[ \begin {array}{ccc} 0&1&3
\\ \noalign{\medskip}1&0&3\\ \noalign{\medskip}3&3&0\end {array}
 \right],  \sigma_7= \left[ \begin {array}{ccc} 0&2&2\\ \noalign{\medskip}2&0&2
\\ \noalign{\medskip}2&2&0\end {array} \right], \sigma_8= \left[ \begin {array}
{ccc} 0&2&2\\ \noalign{\medskip}2&0&3\\ \noalign{\medskip}2&3&0
\end {array} \right],
\]

\[
  \sigma_9= \left[ \begin {array}{ccc} 0&2&3
\\ \noalign{\medskip}2&0&3\\ \noalign{\medskip}3&3&0\end {array}
 \right],  \sigma_{10}=\left[ \begin {array}{ccc} 0&3&3\\ \noalign{\medskip}3&0&3
\\ \noalign{\medskip}3&3&0\end {array} \right]. 
\]

\[
F^{1}_{1}=[1, 1, 1]_{\sigma_1}, F^{1}_{2}=[1, 1, 2]_{\sigma_1}, F^{1}_{3}=[1, 2, 1]_{\sigma_1}, F^{1}_{4}=[2, 1, 1]_{\sigma_1}, F^{1}_{5}=[1, 1, 3]_{\sigma_1}, F^{1}_{6}=[1, 3, 1]_{\sigma_1}
\]
\[
F^{1}_{7}=[3, 1, 1]_{\sigma_1}, F^{1}_{8}=[1, 2, 2]_{\sigma_1}, F^{1}_{9}=[2, 1, 2]_{\sigma_1}, F^{1}_{10}=[2, 2, 1]_{\sigma_1}, F^{1}_{11}=[1, 2, 3]_{\sigma_1}, F^{1}_{12}=[1, 3, 2]_{\sigma_1}
\]
\[
F^{1}_{13}=[2, 1, 3]_{\sigma_1}, F^{1}_{14}=[2, 3, 1]_{\sigma_1}, F^{1}_{15}=[3, 1, 2]_{\sigma_1}, F^{1}_{16}=[3, 2, 1]_{\sigma_1}, F^{1}_{17}=[1, 3, 3]_{\sigma_1}, F^{1}_{18}=[3, 1, 3]_{\sigma_1}
\]
\[
F^{1}_{19}=[3, 3, 1]_{\sigma_1}, F^{1}_{20}=[2, 2, 2]_{\sigma_1}, F^{1}_{21}=[2, 2, 3]_{\sigma_1}, F^{1}_{22}=[2, 3, 2]_{\sigma_1}, F^{1}_{23}=[3, 2, 2]_{\sigma_1}, F^{1}_{24}=[2, 3, 3]_{\sigma_1}
\]
\[
F^{1}_{25}=[3, 2, 3]_{\sigma_1}, F^{1}_{26}=[3, 3, 2]_{\sigma_1}, F^{1}_{27}=[3, 3, 3]_{\sigma_1}, F^{2}_{1}=[1, 1, 1]_{\sigma_2}, F^{2}_{2}=[1, 1, 2]_{\sigma_2}, F^{2}_{3}=[1, 2, 1]_{\sigma_2}
\]
\[
F^{2}_{4}=[1, 1, 3]_{\sigma_2}, F^{2}_{5}=[1, 3, 1]_{\sigma_2}, F^{2}_{6}=[1, 2, 2]_{\sigma_2}, F^{2}_{7}=[1, 2, 3]_{\sigma_2}, F^{2}_{8}=[1, 3, 2]_{\sigma_2}, F^{2}_{9}=[1, 3, 3]_{\sigma_2}
\]
\[
F^{2}_{10}=[2, 1, 1]_{\sigma_2}, F^{2}_{11}=[2, 1, 2]_{\sigma_2}, F^{2}_{12}=[2, 2, 1]_{\sigma_2}, F^{2}_{13}=[2, 1, 3]_{\sigma_2}, F^{2}_{14}=[2, 3, 1]_{\sigma_2}, F^{2}_{15}=[2, 2, 2]_{\sigma_2}
\]
\[
F^{2}_{16}=[2, 2, 3]_{\sigma_2}, F^{2}_{17}=[2, 3, 2]_{\sigma_2}, F^{2}_{18}=[2, 3, 3]_{\sigma_2}, F^{2}_{19}=[3, 1, 1]_{\sigma_2}, F^{2}_{20}=[3, 1, 2]_{\sigma_2}, F^{2}_{21}=[3, 2, 1]_{\sigma_2}
\]
\[
F^{2}_{22}=[3, 1, 3]_{\sigma_2}, F^{2}_{23}=[3, 3, 1]_{\sigma_2}, F^{2}_{24}=[3, 2, 2]_{\sigma_2}, F^{2}_{25}=[3, 2, 3]_{\sigma_2}, F^{2}_{26}=[3, 3, 2]_{\sigma_2}, F^{2}_{27}=[3, 3, 3]_{\sigma_2}
\]
\[
F^{3}_{1}=[1, 1, 1]_{\sigma_3}, F^{3}_{2}=[1, 1, 2]_{\sigma_3}, F^{3}_{3}=[1, 2, 1]_{\sigma_3}, F^{3}_{4}=[1, 1, 3]_{\sigma_3}, F^{3}_{5}=[1, 3, 1]_{\sigma_3}, F^{3}_{6}=[1, 2, 2]_{\sigma_3}
\]
\[
F^{3}_{7}=[1, 2, 3]_{\sigma_3}, F^{3}_{8}=[1, 3, 2]_{\sigma_3}, F^{3}_{9}=[1, 3, 3]_{\sigma_3}, F^{3}_{10}=[2, 1, 1]_{\sigma_3}, F^{3}_{11}=[2, 1, 2]_{\sigma_3}, F^{3}_{12}=[2, 2, 1]_{\sigma_3}
\]
\[
F^{3}_{13}=[2, 1, 3]_{\sigma_3}, F^{3}_{14}=[2, 3, 1]_{\sigma_3}, F^{3}_{15}=[2, 2, 2]_{\sigma_3}, F^{3}_{16}=[2, 2, 3]_{\sigma_3}, F^{3}_{17}=[2, 3, 2]_{\sigma_3}, F^{3}_{18}=[2, 3, 3]_{\sigma_3}
\]
\[
F^{3}_{19}=[3, 1, 1]_{\sigma_3}, F^{3}_{20}=[3, 1, 2]_{\sigma_3}, F^{3}_{21}=[3, 2, 1]_{\sigma_3}, F^{3}_{22}=[3, 1, 3]_{\sigma_3}, F^{3}_{23}=[3, 3, 1]_{\sigma_3}, F^{3}_{24}=[3, 2, 2]_{\sigma_3}
\]
\[
F^{3}_{25}=[3, 2, 3]_{\sigma_3}, F^{3}_{26}=[3, 3, 2]_{\sigma_3}, F^{3}_{27}=[3, 3, 3]_{\sigma_3}, F^{4}_{1}=[1, 1, 1]_{\sigma_4}, F^{4}_{2}=[1, 1, 2]_{\sigma_4}, F^{4}_{3}=[1, 1, 3]_{\sigma_4}
\]
\[
F^{4}_{4}=[1, 2, 1]_{\sigma_4}, F^{4}_{5}=[2, 1, 1]_{\sigma_4}, F^{4}_{6}=[1, 2, 2]_{\sigma_4}, F^{4}_{7}=[2, 1, 2]_{\sigma_4}, F^{4}_{8}=[1, 2, 3]_{\sigma_4}, F^{4}_{9}=[2, 1, 3]_{\sigma_4}
\]
\[
F^{4}_{10}=[1, 3, 1]_{\sigma_4}, F^{4}_{11}=[3, 1, 1]_{\sigma_4}, F^{4}_{12}=[1, 3, 2]_{\sigma_4}, F^{4}_{13}=[3, 1, 2]_{\sigma_4}, F^{4}_{14}=[1, 3, 3]_{\sigma_4}, F^{4}_{15}=[3, 1, 3]_{\sigma_4}
\]
\[
F^{4}_{16}=[2, 2, 1]_{\sigma_4}, F^{4}_{17}=[2, 2, 2]_{\sigma_4}, F^{4}_{18}=[2, 2, 3]_{\sigma_4}, F^{4}_{19}=[2, 3, 1]_{\sigma_4}, F^{4}_{20}=[3, 2, 1]_{\sigma_4}, F^{4}_{21}=[2, 3, 2]_{\sigma_4}
\]
\[
F^{4}_{22}=[3, 2, 2]_{\sigma_4}, F^{4}_{23}=[2, 3, 3]_{\sigma_4}, F^{4}_{24}=[3, 2, 3]_{\sigma_4}, F^{4}_{25}=[3, 3, 1]_{\sigma_4}, F^{4}_{26}=[3, 3, 2]_{\sigma_4}, F^{4}_{27}=[3, 3, 3]_{\sigma_4}
\]
\[
F^{5}_{1}=[1, 1, 1]_{\sigma_5}, F^{5}_{2}=[1, 1, 2]_{\sigma_5}, F^{5}_{3}=[1, 1, 3]_{\sigma_5}, F^{5}_{4}=[1, 2, 1]_{\sigma_5}, F^{5}_{5}=[1, 2, 2]_{\sigma_5},  F^{5}_{6}=[1, 2, 3]_{\sigma_5}
\]
\[
F^{5}_{7}=[2, 1, 1]_{\sigma_5}, F^{5}_{8}=[1, 3, 1]_{\sigma_5}, F^{5}_{9}=[1, 3, 2]_{\sigma_5}, F^{5}_{10}=[1, 3, 3]_{\sigma_5}, F^{5}_{11}=[2, 1, 2]_{\sigma_5}, F^{5}_{12}=[2, 1, 3]_{\sigma_5}
\]
\[
F^{5}_{13}=[3, 1, 1]_{\sigma_5}, F^{5}_{14}=[3, 1, 2]_{\sigma_5}, F^{5}_{15}=[3, 1, 3]_{\sigma_5}, F^{5}_{16}=[2, 2, 1]_{\sigma_5}, F^{5}_{17}=[2, 2, 2]_{\sigma_5}, F^{5}_{18}=[2, 2, 3]_{\sigma_5}
\]
\[
F^{5}_{19}=[2, 3, 1]_{\sigma_5}, F^{5}_{20}=[2, 3, 2]_{\sigma_5}, F^{5}_{21}=[2, 3, 3]_{\sigma_5}, F^{5}_{22}=[3, 2, 1]_{\sigma_5}, F^{5}_{23}=[3, 2, 2]_{\sigma_5}, F^{5}_{24}=[3, 2, 3]_{\sigma_5}
\]
\[
F^{5}_{25}=[3, 3, 1]_{\sigma_5}, F^{5}_{26}=[3, 3, 2]_{\sigma_5}, F^{5}_{27}=[3, 3, 3]_{\sigma_5}, F^{6}_{1}=[1, 1, 1]_{\sigma_6}, F^{6}_{2}=[1, 1, 2]_{\sigma_6}, F^{6}_{3}=[1, 2, 1]_{\sigma_6}
\]
\[
F^{6}_{4}=[2, 1, 1]_{\sigma_6}, F^{6}_{5}=[1, 1, 3]_{\sigma_6}, F^{6}_{6}=[1, 2, 2]_{\sigma_6}, F^{6}_{7}=[2, 1, 2]_{\sigma_6}, F^{6}_{8}=[1, 2, 3]_{\sigma_6}, F^{6}_{9}=[2, 1, 3]_{\sigma_6}
\]
\[
F^{6}_{10}=[1, 3, 1]_{\sigma_6}, F^{6}_{11}=[3, 1, 1]_{\sigma_6}, F^{6}_{12}=[1, 3, 2]_{\sigma_6}, F^{6}_{13}=[3, 1, 2]_{\sigma_6}, F^{6}_{14}=[1, 3, 3]_{\sigma_6}, F^{6}_{15}=[3, 1, 3]_{\sigma_6}
\]
\[
F^{6}_{16}=[2, 2, 1]_{\sigma_6}, F^{6}_{17}=[2, 2, 2]_{\sigma_6}, F^{6}_{18}=[2, 2, 3]_{\sigma_6}, F^{6}_{19}=[2, 3, 1]_{\sigma_6}, F^{6}_{20}=[3, 2, 1]_{\sigma_6}, F^{6}_{21}=[2, 3, 2]_{\sigma_6}
\]
\[
F^{6}_{22}=[3, 2, 2]_{\sigma_6}, F^{6}_{23}=[2, 3, 3]_{\sigma_6}, F^{6}_{24}=[3, 2, 3]_{\sigma_6}, F^{6}_{25}=[3, 3, 1]_{\sigma_6}, F^{6}_{26}=[3, 3, 2]_{\sigma_6}, F^{6}_{27}=[3, 3, 3]_{\sigma_6}
\]
\[
F^{7}_{1}=[1, 1, 1]_{\sigma_7}, F^{7}_{2}=[1, 1, 2]_{\sigma_7}, F^{7}_{3}=[1, 2, 1]_{\sigma_7}, F^{7}_{4}=[2, 1, 1]_{\sigma_7}, F^{7}_{5}=[1, 1, 3]_{\sigma_7}, F^{7}_{6}=[1, 3, 1]_{\sigma_7}
\]
\[
F^{7}_{7}=[3, 1, 1]_{\sigma_7}, F^{7}_{8}=[1, 2, 2]_{\sigma_7}, F^{7}_{9}=[2, 1, 2]_{\sigma_7}, F^{7}_{10}=[2, 2, 1]_{\sigma_7}, F^{7}_{11}=[1, 2, 3]_{\sigma_7}, F^{7}_{12}=[1, 3, 2]_{\sigma_7}
\]
\[
F^{7}_{13}=[2, 1, 3]_{\sigma_7}, F^{7}_{14}=[2, 3, 1]_{\sigma_7}, F^{7}_{15}=[3, 1, 2]_{\sigma_7}, F^{7}_{16}=[3, 2, 1]_{\sigma_7}, F^{7}_{17}=[1, 3, 3]_{\sigma_7}, F^{7}_{18}=[3, 1, 3]_{\sigma_7}
\]
\[
F^{7}_{19}=[3, 3, 1]_{\sigma_7}, F^{7}_{20}=[2, 2, 2]_{\sigma_7}, F^{7}_{21}=[2, 2, 3]_{\sigma_7}, F^{7}_{22}=[2, 3, 2]_{\sigma_7}, F^{7}_{23}=[3, 2, 2]_{\sigma_7}, F^{7}_{24}=[2, 3, 3]_{\sigma_7}
\]
\[
F^{7}_{25}=[3, 2, 3]_{\sigma_7}, F^{7}_{26}=[3, 3, 2]_{\sigma_7}, F^{7}_{27}=[3, 3, 3]_{\sigma_7}, F^{8}_{1}=[1, 1, 1]_{\sigma_8}, F^{8}_{2}=[1, 1, 2]_{\sigma_8}, F^{8}_{3}=[1, 2, 1]_{\sigma_8}
\]
\[
F^{8}_{4}=[2, 1, 1]_{\sigma_8}, F^{8}_{5}=[1, 1, 3]_{\sigma_8}, F^{8}_{6}=[1, 3, 1]_{\sigma_8}, F^{8}_{7}=[3, 1, 1]_{\sigma_8}, F^{8}_{8}=[1, 2, 2]_{\sigma_8}, F^{8}_{9}=[2, 1, 2]_{\sigma_8}
\]
\[
F^{8}_{10}=[2, 2, 1]_{\sigma_8}, F^{8}_{11}=[1, 2, 3]_{\sigma_8}, F^{8}_{12}=[1, 3, 2]_{\sigma_8}, F^{8}_{13}=[1, 3, 3]_{\sigma_8}, F^{8}_{14}=[2, 1, 3]_{\sigma_8}, F^{8}_{15}=[2, 3, 1]_{\sigma_8}
\]
\[
F^{8}_{16}=[3, 1, 2]_{\sigma_8}, F^{8}_{17}=[3, 2, 1]_{\sigma_8}, F^{8}_{18}=[3, 1, 3]_{\sigma_8}, F^{8}_{19}=[3, 3, 1]_{\sigma_8}, F^{8}_{20}=[2, 2, 2]_{\sigma_8}, F^{8}_{21}=[2, 2, 3]_{\sigma_8}
\]
\[
F^{8}_{22}=[2, 3, 2]_{\sigma_8}, F^{8}_{23}=[2, 3, 3]_{\sigma_8}, F^{8}_{24}=[3, 2, 2]_{\sigma_8}, F^{8}_{25}=[3, 2, 3]_{\sigma_8}, F^{8}_{26}=[3, 3, 2]_{\sigma_8}, F^{8}_{27}=[3, 3, 3]_{\sigma_8}
\]
\[
F^{9}_{1}=[1, 1, 1]_{\sigma_9}, F^{9}_{2}=[1, 1, 2]_{\sigma_9}, F^{9}_{3}=[1, 2, 1]_{\sigma_9}, F^{9}_{4}=[2, 1, 1]_{\sigma_9}, F^{9}_{5}=[1, 1, 3]_{\sigma_9}, F^{9}_{6}=[1, 3, 1]_{\sigma_9}
\]
\[
F^{9}_{7}=[3, 1, 1]_{\sigma_9}, F^{9}_{8}=[1, 2, 2]_{\sigma_9}, F^{9}_{9}=[2, 1, 2]_{\sigma_9}, F^{9}_{10}=[2, 2, 1]_{\sigma_9}, F^{9}_{11}=[1, 2, 3]_{\sigma_9}, F^{9}_{12}=[2, 1, 3]_{\sigma_9}
\]
\[
F^{9}_{13}=[1, 3, 2]_{\sigma_9}, F^{9}_{14}=[3, 1, 2]_{\sigma_9}, F^{9}_{15}=[2, 3, 1]_{\sigma_9}, F^{9}_{16}=[3, 2, 1]_{\sigma_9}, F^{9}_{17}=[1, 3, 3]_{\sigma_9}, F^{9}_{18}=[3, 1, 3]_{\sigma_9}
\]
\[
F^{9}_{19}=[3, 3, 1]_{\sigma_9}, F^{9}_{20}=[2, 2, 2]_{\sigma_9}, F^{9}_{21}=[2, 2, 3]_{\sigma_9}, F^{9}_{22}=[2, 3, 2]_{\sigma_9}, F^{9}_{23}=[3, 2, 2]_{\sigma_9}, F^{9}_{24}=[2, 3, 3]_{\sigma_9}
\]
\[
F^{9}_{25}=[3, 2, 3]_{\sigma_9}, F^{9}_{26}=[3, 3, 2]_{\sigma_9}, F^{9}_{27}=[3, 3, 3]_{\sigma_9}, F^{10}_{1}=[1, 1, 1]_{\sigma_{10}}, F^{10}_{2}=[1, 1, 2]_{\sigma_{10}}, F^{10}_{3}=[1, 2, 1]_{\sigma_{10}}
\]
\[
F^{10}_{4}=[2, 1, 1]_{\sigma_{10}}, F^{10}_{5}=[1, 1, 3]_{\sigma_{10}}, F^{10}_{6}=[1, 3, 1]_{\sigma_{10}}, F^{10}_{7}=[3, 1, 1]_{\sigma_{10}}, F^{10}_{8}=[1, 2, 2]_{\sigma_{10}}, F^{10}_{9}=[2, 1, 2]_{\sigma_{10}}
\]
\[
  F^{10}_{10}=[2, 2, 1]_{\sigma_{10}}, F^{10}_{11}=[1, 2, 3]_{\sigma_{10}}, F^{10}_{12}=[1, 3, 2]_{\sigma_{10}}, F^{10}_{13}=[2, 1, 3]_{\sigma_{10}}, F^{10}_{14}=[2, 3, 1]_{\sigma_{10}}, F^{10}_{15}=[3, 1, 2]_{\sigma_{10}}
\]
\[
F^{10}_{16}=[3, 2, 1]_{\sigma_{10}}, F^{10}_{17}=[1, 3, 3]_{\sigma_{10}}, F^{10}_{18}=[3, 1, 3]_{\sigma_{10}}, F^{10}_{19}=[3, 3, 1]_{\sigma_{10}}, F^{10}_{20}=[2, 2, 2]_{\sigma_{10}}, F^{10}_{21}=[2, 2, 3]_{\sigma_{10}}
\]
\[
F^{10}_{22}=[2, 3, 2]_{\sigma_{10}}, F^{10}_{23}=[3, 2, 2]_{\sigma_{10}}, F^{10}_{24}=[2, 3, 3]_{\sigma_{10}}, F^{10}_{25}=[3, 2, 3]_{\sigma_{10}}, F^{10}_{26}=[3, 3, 2]_{\sigma_{10}}, F^{10}_{27}=[3, 3, 3]_{\sigma_{10}}
\]

\newcommand\scalemath[2]{\scalebox{#1}{\mbox{\ensuremath{\displaystyle #2}}}}

\[\scalemath{0.28}{Q^1=\left(
\begin{array}{rrrrrrrrrrrrrrrrrrrrrrrrrrr}
 \frac{24}{25} & 0 & 0 & 0 & 0 & 0 & 0 & 0 & 0 & 0 & 0 & 0 & 0 & 0 & 0 & 0 & 0 & 0 & 0 &  -\frac{12}{25} & 0 & 0 & 0 & 0 & 0 & 0 & 0 \\[1.5em]  
 0 & 0 & 0 & 0 & 0 & 0 & 0 & 0 & 0 & 0 & 0 & 0 & 0 & 0 & 0 & 0 & 0 & 0 & 0 & 0 & 0 & 0 & 0 & 0 & 0 & 0 & 0 \\[1.5em]
 0 & 0 & 0 & 0 & 0 & 0 & 0 & 0 & 0 & 0 & 0 & 0 & 0 & 0 & 0 & 0 & 0 & 0 & 0 & 0 & 0 & 0 & 0 & 0 & 0 & 0 & 0 \\[1.5em]
 0 & 0 & 0 & 0 & 0 & 0 & 0 & 0 & 0 & 0 & 0 & 0 & 0 & 0 & 0 & 0 & 0 & 0 & 0 & 0 & 0 & 0 & 0 & 0 & 0 & 0 & 0 \\[1.5em]
 0 & 0 & 0 & 0 & 0 & 0 & 0 & 0 & 0 & 0 & 0 & 0 & 0 & 0 & 0 & 0 & 0 & 0 & 0 & 0 & 0 & 0 & 0 & 0 & 0 & 0 & 0 \\[1.5em]
 0 & 0 & 0 & 0 & 0 & 0 & 0 & 0 & 0 & 0 & 0 & 0 & 0 & 0 & 0 & 0 & 0 & 0 & 0 & 0 & 0 & 0 & 0 & 0 & 0 & 0 & 0 \\[1.5em]
 0 & 0 & 0 & 0 & 0 & 0 & 0 & 0 & 0 & 0 & 0 & 0 & 0 & 0 & 0 & 0 & 0 & 0 & 0 & 0 & 0 & 0 & 0 & 0 & 0 & 0 & 0 \\[1.5em]
 0 & 0 & 0 & 0 & 0 & 0 & 0 & 0 & 0 & 0 & 0 & 0 & 0 & 0 & 0 & 0 & 0 & 0 & 0 & 0 & 0 & 0 & 0 & 0 & 0 & 0 & 0 \\
\end{array} \right)}
\]

\[ \scalemath{0.28}{Q^2=\left(
\begin{array}{rrrrrrrrrrrrrrrrrrrrrrrrrrr}
 \frac{2778017}{500000} & \frac{264631}{250000} & \frac{264631}{250000} & \frac{189083}{250000} & \frac{189083}{250000} & \frac{84493}{500000} & \frac{2466747}{2000000} & \frac{2466747}{2000000} & \frac{29541}{40000} & -\frac{147603}{200000} & -\frac{4029}{6250} & -\frac{4029}{6250} & -\frac{706289}{1000000} & -\frac{706289}{1000000} & -\frac{72793}{1000000} & \frac{37619}{200000} & \frac{37619}{200000} & \frac{752127}{1000000} & -\frac{1177467}{1000000} & \frac{21949}{500000} & \frac{21949}{500000} & -\frac{432183}{500000} & -\frac{432183}{500000} & -\frac{452251}{1000000} & -\frac{32199}{20000} & -\frac{32199}{20000} & \frac{1007971}{1000000} \\[1.5em]
 \frac{264631}{250000} & \frac{1175347}{500000} & \frac{178747}{500000} & \frac{14623}{10000} & -\frac{87887}{100000} & -\frac{350373}{400000} & \frac{1206007}{2000000} & -\frac{2238327}{2000000} & \frac{81507}{125000} & \frac{240839}{2000000} & \frac{106569}{125000} & -\frac{113991}{250000} & \frac{338627}{1000000} & -\frac{19353}{8000} & -\frac{725869}{1000000} & -\frac{20517}{50000} & -\frac{5831}{20000} & \frac{1822}{3125} & \frac{29893}{250000} & \frac{822849}{1000000} & \frac{668611}{1000000} & \frac{3060157}{4000000} & -\frac{295439}{800000} & -\frac{593917}{1000000} & -\frac{902381}{1000000} & -\frac{58731}{40000} & -\frac{3223}{200000} \\
\end{array} \right)}
\]

\[ \scalemath{0.28}{Q^3=\left(
\begin{array}{rrrrrrrrrrrrrrrrrrrrrrrrrrr} 
 \frac{5565837}{1000000} & \frac{56579}{62500} & \frac{56579}{62500} & \frac{95841}{100000} & \frac{95841}{100000} & \frac{26931}{31250} & \frac{1282331}{1000000} & \frac{1282331}{1000000} & \frac{146499}{200000} & -\frac{439633}{250000} & -\frac{2126197}{2000000} & -\frac{2126197}{2000000} & -\frac{251997}{1000000} & -\frac{251997}{1000000} & \frac{60271}{100000} & -\frac{1552571}{1000000} & -\frac{1552571}{1000000} & -\frac{778727}{500000} & -\frac{87693}{200000} & -\frac{123057}{200000} & -\frac{123057}{200000} & -\frac{80613}{200000} & -\frac{80613}{200000} & \frac{479049}{1000000} & \frac{98629}{250000} & \frac{98629}{250000} & \frac{556363}{1000000} \\[1.5em]
 \frac{56579}{62500} & \frac{4286203}{1000000} & -\frac{1381877}{1000000} & \frac{1371247}{1000000} & -\frac{900013}{1000000} & \frac{553277}{500000} & -\frac{1867829}{2000000} & \frac{4872519}{2000000} & -\frac{210997}{200000} & -\frac{557023}{1000000} & \frac{234477}{500000} & -\frac{833759}{500000} & \frac{486533}{1000000} & -\frac{1549677}{1000000} & \frac{161847}{200000} & -\frac{246501}{200000} & \frac{105977}{1000000} & -\frac{309197}{500000} & -\frac{303781}{1000000} & \frac{20182}{15625} & -\frac{117243}{62500} & \frac{257799}{1000000} & -\frac{1203121}{1000000} & \frac{1205267}{1000000} & -\frac{136499}{500000} & \frac{31893}{125000} & -\frac{1010729}{1000000} \\
\end{array} \right)}
\]

\[ \scalemath{0.28}{Q^4=\left(
\begin{array}{rrrrrrrrrrrrrrrrrrrrrrrrrrr}
 \frac{5440101}{1000000} & \frac{864653}{500000} & \frac{70733}{500000} & \frac{25381}{50000} & \frac{25381}{50000} & -\frac{174431}{500000} & -\frac{174431}{500000} & -\frac{103323}{1000000} & -\frac{103323}{1000000} & \frac{91497}{500000} & \frac{91497}{500000} & -\frac{4311}{25000} & -\frac{4311}{25000} & \frac{249313}{500000} & \frac{249313}{500000} & \frac{38057}{100000} & \frac{556363}{1000000} & -\frac{16759}{25000} & \frac{530347}{500000} & \frac{530347}{500000} & -\frac{1010729}{1000000} & -\frac{1010729}{1000000} & -\frac{3569}{200000} & -\frac{3569}{200000} & -\frac{1391989}{500000} & -\frac{299569}{500000} & -\frac{420189}{500000} \\[1.5em]
 \frac{864653}{500000} & \frac{69}{25} & \frac{140931}{1000000} & \frac{755093}{2000000} & \frac{755093}{2000000} & -\frac{1697897}{2000000} & -\frac{1697897}{2000000} & -\frac{261017}{1000000} & -\frac{261017}{1000000} & \frac{481571}{1000000} & \frac{481571}{1000000} & -\frac{349949}{2000000} & -\frac{349949}{2000000} & \frac{556887}{500000} & \frac{556887}{500000} & \frac{76291}{1000000} & \frac{146499}{200000} & -\frac{50643}{125000} & \frac{133969}{250000} & \frac{133969}{250000} & -\frac{210997}{200000} & -\frac{210997}{200000} & \frac{136829}{1000000} & \frac{136829}{1000000} & -\frac{889773}{500000} & -\frac{326801}{250000} & -\frac{1123943}{1000000} \\
\end{array} \right)}
\]

\[ \scalemath{0.28}{Q^5=\left(
\begin{array}{rrrrrrrrrrrrrrrrrrrrrrrrrrr} 
 \frac{10514917}{1000000} & \frac{2069571}{1000000} & \frac{447947}{200000} & \frac{2261771}{500000} & \frac{114561}{250000} & -\frac{52361}{31250} & \frac{1699089}{1000000} & \frac{115293}{62500} & -\frac{166623}{125000} & -\frac{218219}{100000} & -\frac{200607}{125000} & -\frac{1493077}{1000000} & \frac{894259}{200000} & -\frac{304791}{250000} & \frac{794939}{1000000} & \frac{409309}{1000000} & -\frac{1547491}{1000000} & -\frac{3670483}{1000000} & \frac{3160943}{500000} & -\frac{2429933}{1000000} & -\frac{701231}{250000} & \frac{610697}{500000} & -\frac{1056679}{500000} & -\frac{2139699}{1000000} & \frac{128627}{250000} & -\frac{1573677}{500000} & -\frac{1547491}{1000000} \\[1.5em]
 \frac{2069571}{1000000} & \frac{3893387}{500000} & \frac{326827}{100000} & \frac{1176901}{500000} & \frac{400801}{125000} & \frac{2667}{40000} & -\frac{62923}{50000} & \frac{490019}{500000} & \frac{1693083}{500000} & -\frac{1043}{31250} & -\frac{2774933}{1000000} & \frac{220301}{250000} & -\frac{194949}{1000000} & -\frac{2809991}{1000000} & -\frac{909}{250000} & -\frac{2257659}{1000000} & -\frac{218219}{100000} & -\frac{1671517}{1000000} & -\frac{36881}{500000} & -\frac{1043}{31250} & \frac{263959}{500000} & -\frac{1118011}{1000000} & -\frac{1637251}{1000000} & -\frac{1140441}{500000} & -\frac{438149}{200000} & -\frac{216621}{125000} & -\frac{2429933}{1000000} \\
\end{array} \right)}
\]

\[ \scalemath{0.28}{Q^6=\left(
\begin{array}{rrrrrrrrrrrrrrrrrrrrrrrrrrr}
 \frac{1199763}{250000} & -\frac{285803}{1000000} & \frac{71391}{250000} & \frac{71391}{250000} & \frac{349831}{250000} & \frac{29723}{50000} & \frac{29723}{50000} & -\frac{120113}{2000000} & -\frac{120113}{2000000} & \frac{16989}{40000} & \frac{16989}{40000} & -\frac{315343}{1000000} & -\frac{315343}{1000000} & -\frac{725869}{1000000} & -\frac{725869}{1000000} & -\frac{73689}{31250} & -\frac{1809}{3125} & -\frac{54581}{100000} & \frac{4883}{4000} & \frac{4883}{4000} & -\frac{33459}{200000} & -\frac{33459}{200000} & -\frac{23927}{31250} & -\frac{23927}{31250} & \frac{454703}{1000000} & -\frac{729337}{1000000} & -\frac{72793}{1000000} \\[1.5em]
 -\frac{285803}{1000000} & \frac{33949}{15625} & -\frac{2509}{62500} & -\frac{2509}{62500} & \frac{169727}{1000000} & \frac{716943}{1000000} & \frac{716943}{1000000} & -\frac{143887}{400000} & -\frac{143887}{400000} & -\frac{175417}{1000000} & -\frac{175417}{1000000} & \frac{95311}{250000} & \frac{95311}{250000} & -\frac{593917}{1000000} & -\frac{593917}{1000000} & -\frac{193021}{500000} & \frac{18507}{200000} & -\frac{542951}{500000} & \frac{89927}{500000} & \frac{89927}{500000} & \frac{60379}{125000} & \frac{60379}{125000} & -\frac{430067}{500000} & -\frac{430067}{500000} & -\frac{246629}{1000000} & \frac{1070269}{1000000} & -\frac{452251}{1000000} \\[1.5em]
 \frac{71391}{250000} & -\frac{2509}{62500} & \frac{647517}{500000} & -\frac{3743}{4000} & \frac{93557}{250000} & \frac{993357}{1000000} & -\frac{642833}{1000000} & \frac{866061}{1000000} & -\frac{402801}{1000000} & -\frac{344547}{1000000} & \frac{131257}{1000000} & -\frac{45807}{125000} & -\frac{15833}{125000} & -\frac{5831}{20000} & -\frac{20517}{50000} & -\frac{6009}{50000} & \frac{51343}{500000} & -\frac{501}{3125} & \frac{189361}{200000} & -\frac{208783}{1000000} & \frac{3577}{100000} & -\frac{216203}{500000} & -\frac{47049}{200000} & \frac{269899}{1000000} & -\frac{1041}{15625} & -\frac{24997}{100000} & \frac{37619}{200000} \\
\end{array} \right)}
\]

\[ \scalemath{0.28}{Q^7=\left(
\begin{array}{rrrrrrrrrrrrrrrrrrrrrrrrrrr} 
 0 & 0 & 0 & 0 & 0 & 0 & 0 & 0 & 0 & 0 & 0 & 0 & 0 & 0 & 0 & 0 & 0 & 0 & 0 & 0 & 0 & 0 & 0 & 0 & 0 & 0 & 0 \\[1.5em]
 0 & 0 & 0 & 0 & 0 & 0 & 0 & 0 & 0 & 0 & 0 & 0 & 0 & 0 & 0 & 0 & 0 & 0 & 0 & 0 & 0 & 0 & 0 & 0 & 0 & 0 & 0 \\[1.5em]
 0 & 0 & 0 & 0 & 0 & 0 & 0 & 0 & 0 & 0 & 0 & 0 & 0 & 0 & 0 & 0 & 0 & 0 & 0 & 0 & 0 & 0 & 0 & 0 & 0 & 0 & 0 \\[1.5em]
 0 & 0 & 0 & 0 & 0 & 0 & 0 & 0 & 0 & 0 & 0 & 0 & 0 & 0 & 0 & 0 & 0 & 0 & 0 & 0 & 0 & 0 & 0 & 0 & 0 & 0 & 0 \\[1.5em]
 0 & 0 & 0 & 0 & 0 & 0 & 0 & 0 & 0 & 0 & 0 & 0 & 0 & 0 & 0 & 0 & 0 & 0 & 0 & 0 & 0 & 0 & 0 & 0 & 0 & 0 & 0 \\[1.5em]
 0 & 0 & 0 & 0 & 0 & 0 & 0 & 0 & 0 & 0 & 0 & 0 & 0 & 0 & 0 & 0 & 0 & 0 & 0 & 0 & 0 & 0 & 0 & 0 & 0 & 0 & 0 \\[1.5em]
 0 & 0 & 0 & 0 & 0 & 0 & 0 & 0 & 0 & 0 & 0 & 0 & 0 & 0 & 0 & 0 & 0 & 0 & 0 & 0 & 0 & 0 & 0 & 0 & 0 & 0 & 0 \\[1.5em]
 0 & 0 & 0 & 0 & 0 & 0 & 0 & 0 & 0 & 0 & 0 & 0 & 0 & 0 & 0 & 0 & 0 & 0 & 0 & 0 & 0 & 0 & 0 & 0 & 0 & 0 & 0 \\
\end{array} \right)}
\]

\[ \scalemath{0.28}{Q^8=\left(
\begin{array}{rrrrrrrrrrrrrrrrrrrrrrrrrrr} 
 \frac{54291}{50000} & -\frac{602077}{1000000} & -\frac{602077}{1000000} & \frac{449617}{1000000} & -\frac{106577}{200000} & -\frac{106577}{200000} & \frac{936027}{1000000} & -\frac{1881}{5000} & \frac{134077}{200000} & \frac{134077}{200000} & -\frac{143673}{500000} & -\frac{143673}{500000} & \frac{18507}{200000} & \frac{413803}{500000} & \frac{413803}{500000} & \frac{423961}{1000000} & \frac{423961}{1000000} & \frac{51343}{500000} & \frac{51343}{500000} & \frac{1007971}{1000000} & -\frac{3223}{200000} & -\frac{3223}{200000} & -\frac{708179}{1000000} & -\frac{225529}{500000} & -\frac{319753}{500000} & -\frac{319753}{500000} & -\frac{1809}{3125} \\[1.5em]
 -\frac{602077}{1000000} & \frac{2988803}{2000000} & -\frac{395577}{2000000} & -\frac{1049941}{2000000} & \frac{1288801}{4000000} & \frac{55699}{800000} & -\frac{26879}{62500} & \frac{191689}{250000} & \frac{931017}{2000000} & -\frac{3413871}{2000000} & -\frac{307867}{1000000} & \frac{310813}{200000} & \frac{60379}{125000} & \frac{124953}{4000000} & -\frac{6444879}{4000000} & \frac{1023459}{2000000} & -\frac{2687641}{2000000} & -\frac{216203}{500000} & \frac{3577}{100000} & -\frac{432183}{500000} & -\frac{295439}{800000} & \frac{3060157}{4000000} & \frac{76923}{1000000} & \frac{566737}{2000000} & \frac{9991}{50000} & \frac{1226811}{2000000} & -\frac{33459}{200000} \\
\end{array} \right)}
\]

\[ \scalemath{0.28}{Q^9=\left(
\begin{array}{rrrrrrrrrrrrrrrrrrrrrrrrrrr} 
 \frac{141293}{125000} & \frac{581779}{500000} & -\frac{622593}{1000000} & -\frac{622593}{1000000} & \frac{335359}{500000} & -\frac{307201}{500000} & -\frac{307201}{500000} & \frac{11977}{200000} & \frac{11977}{200000} & -\frac{54569}{250000} & \frac{48601}{62500} & \frac{48601}{62500} & \frac{334931}{1000000} & \frac{334931}{1000000} & -\frac{162363}{500000} & -\frac{162363}{500000} & \frac{161847}{200000} & \frac{161847}{200000} & -\frac{350691}{1000000} & -\frac{420189}{500000} & -\frac{1123943}{1000000} & -\frac{1204541}{2000000} & -\frac{1204541}{2000000} & \frac{69269}{2000000} & \frac{69269}{2000000} & -\frac{81607}{250000} & \frac{60271}{100000} \\[1.5em]
 \frac{581779}{500000} & \frac{1052953}{500000} & -\frac{856587}{1000000} & -\frac{856587}{1000000} & \frac{1093093}{1000000} & -\frac{965499}{2000000} & -\frac{965499}{2000000} & -\frac{111331}{1000000} & -\frac{111331}{1000000} & -\frac{93819}{200000} & \frac{267699}{400000} & \frac{267699}{400000} & -\frac{486669}{1000000} & -\frac{486669}{1000000} & -\frac{220239}{1000000} & -\frac{220239}{1000000} & \frac{1205267}{1000000} & \frac{1205267}{1000000} & -\frac{104109}{1000000} & -\frac{1391989}{500000} & -\frac{889773}{500000} & -\frac{617193}{1000000} & -\frac{617193}{1000000} & \frac{421553}{1000000} & \frac{421553}{1000000} & -\frac{93249}{250000} & \frac{479049}{1000000} \\
\end{array} \right)}
\]

\[ \scalemath{0.28}{Q^{10}=\left(
\begin{array}{rrrrrrrrrrrrrrrrrrrrrrrrrrr} 
 \frac{6}{25} & 0 & 0 & 0 & 0 & 0 & 0 & 0 & 0 & 0 & 0 & 0 & 0 & 0 & 0 & 0 & 0 & 0 & 0 & 0 & 0 & 0 & 0 & 0 & 0 & 0 & -\frac{12}{25} \\[1.5em]
 0 & 0 & 0 & 0 & 0 & 0 & 0 & 0 & 0 & 0 & 0 & 0 & 0 & 0 & 0 & 0 & 0 & 0 & 0 & 0 & 0 & 0 & 0 & 0 & 0 & 0 & 0 \\[1.5em]
 0 & 0 & 0 & 0 & 0 & 0 & 0 & 0 & 0 & 0 & 0 & 0 & 0 & 0 & 0 & 0 & 0 & 0 & 0 & 0 & 0 & 0 & 0 & 0 & 0 & 0 & 0 \\[1.5em]
 0 & 0 & 0 & 0 & 0 & 0 & 0 & 0 & 0 & 0 & 0 & 0 & 0 & 0 & 0 & 0 & 0 & 0 & 0 & 0 & 0 & 0 & 0 & 0 & 0 & 0 & 0 \\[1.5em]
 0 & 0 & 0 & 0 & 0 & 0 & 0 & 0 & 0 & 0 & 0 & 0 & 0 & 0 & 0 & 0 & 0 & 0 & 0 & 0 & 0 & 0 & 0 & 0 & 0 & 0 & 0 \\[1.5em]
 0 & 0 & 0 & 0 & 0 & 0 & 0 & 0 & 0 & 0 & 0 & 0 & 0 & 0 & 0 & 0 & 0 & 0 & 0 & 0 & 0 & 0 & 0 & 0 & 0 & 0 & 0 \\[1.5em]
 0 & 0 & 0 & 0 & 0 & 0 & 0 & 0 & 0 & 0 & 0 & 0 & 0 & 0 & 0 & 0 & 0 & 0 & 0 & 0 & 0 & 0 & 0 & 0 & 0 & 0 & 0 \\[1.5em]
 0 & 0 & 0 & 0 & 0 & 0 & 0 & 0 & 0 & 0 & 0 & 0 & 0 & 0 & 0 & 0 & 0 & 0 & 0 & 0 & 0 & 0 & 0 & 0 & 0 & 0 & 0 \\
\end{array} \right)}
\]

\end{document}